\documentclass{amsart}
\usepackage{amsmath,amsthm}
\usepackage{amsfonts,amssymb}

\usepackage{enumerate}
\usepackage{multirow}

\newtheorem{thm}{Theorem}[section]
\newtheorem{cor}[thm]{Corollary}
\newtheorem{lem}[thm]{Lemma}
\newtheorem{prop}[thm]{Proposition}

\theoremstyle{remark}

 \def\a{{\alpha}} 
 \def\b{{\beta}}
 \def\g{{\gamma}}
 
 \def\t{{\theta}}
 \def\l{{\lambda}}
 \def\d{{\delta}}
 
 \def\s{{\sigma}}
 \def\la{{\langle}}
 \def\ra{{\rangle}}

 \def\xb{{\mathbf x}}

 \def\C{{\mathcal C}}

 \def\CI{{\mathcal I}}
 
 \def\CK{{\mathcal K}}
 \def\CL{{\mathcal L}}
 \def\CO{{\mathcal O}}
 \def\CP{{\mathcal P}}

 \def\CV{{\mathcal V}}
 \def\CW{{\mathcal W}}

 \def\NN{{\mathbb N}}

 \def\RR{{\mathbb R}}
  
 \def\ZZ{{\mathbb Z}}

\newcommand{\wh}{\widehat}

\newif\ifpdf
\ifx\pdfoutput\undefined
  \pdffalse
\else
  \pdftrue
\fi

\ifpdf
  \usepackage[pdftex]{graphicx}
  \DeclareGraphicsExtensions{.pdf,.jpg,.png}
\else
  \usepackage{graphicx}
\fi

\begin{document}
 
\title[Minimal cubature rules and interpolation]
{Minimal Cubature rules and polynomial interpolation 
%for  a family of weight functions 
in two variables}
\author{Yuan Xu}
\address{Department of Mathematics\\ University of Oregon\\
    Eugene, Oregon 97403-1222.}\email{yuan@math.uoregon.edu}

\date{\today}

\begin{abstract}
Minimal cubature rules of degree $4n-1$ for the weight functions 
$$
W_{\a,\b,\pm \frac12}(x,y) = |x+y|^{2\a+1} |x-y|^{2\b+1} ((1-x^2)(1-y^2))^{\pm \frac12}
$$
on $[-1,1]^2$ are constructed explicitly and are shown to be closed related to 
the Gaussian cubature rules in a domain bounded by two lines and a parabola. 
Lagrange interpolation polynomials on the nodes of these cubature rules are constructed
and their Lebesgue constants are determined. 
\end{abstract}

\maketitle

\section{Introduction}
\setcounter{equation}{0}
 
Minimal cubature rules have the smallest number of nodes among all cubature rules
of the same precision. Let $W$ be a non-negative weight function on a domain 
$\Omega \subset \RR^2$. For a positive integer $s$, a cubature rule of precision $s$ 
with respect to $W$ is a finite sum that satisfies
\begin{equation} \label{cuba-generic}
        \int_\Omega f(x,y) W(x,y) dxdy = \sum_{k=1}^N \l_k f(x_k,y_k), 
             \qquad \forall f\in \Pi_{s}^2, 
\end{equation}
where $\Pi_s^2$ denotes the space of polynomials of degree at most $s$ in two variables, 
and there exists at least one function $f^*$ in $\Pi_{s+1}^2$ such that the equation 
\eqref{cuba-generic} does not hold. 
 
It is known that the number of nodes $N$ of a cubature rule necessarily satisfies 
\begin{equation}\label{lbdGaussian}
   N \ge \dim \Pi_{n-1}^2 =  \frac{n (n+1)}{2}, \qquad \hbox{$s = 2n-1$ or $2n-2$}
\end{equation}
(cf. \cite{My, St}). A cubature rule of degree $s$ with $N$ attaining the lower bound in 
\eqref{lbdGaussian} is called Gaussian. Unless quadrature rules in one variable, 
Gaussian cubature rules rarely exist. At the moment, they are known to exist only in 
two cases. The first case, discovered in \cite{SX}, is for a family of weight functions 
that includes, in particular, $W_{\a,\b,\pm 12}$ defined by 
\begin{equation} \label{Wab}
    W_{\a,\b,\pm \frac12}(u,v) = (1-u+v)^\a (1+u+v)^\b (u^2- 4 v)^{\pm \frac12}
\end{equation}
on the domain $\Omega  = \{(u,v): 1+u+v > 0, 1-u+v > 0, u^2 > 4v\}$, bounded by 
two lines and a parabola. On the other hand, Gaussian cubature rules of degree 
$2n-1$ do not exist when $W$ is centrally symmetric, that is, when $W$ and its
domain $\Omega$ are both symmetric with respect to the origin: 
$(-x,-y) \in \Omega$ whenever $(x,y)\in \Omega$ and $W(-x,-y) = W(x,y)$. 
For centrally symmetric weight functions and $s = 2n-1$, a stronger lower bound
\cite{M} for the number of nodes is given by 
\begin{equation}\label{lwbd}
   N \ge \dim \Pi_{n-1}^2 + \left \lfloor \frac{n}{2} \right \rfloor 
            = \frac{n(n+1)}{2} +  \left \lfloor \frac{n}{2} \right \rfloor. 
\end{equation}
A cubature rule that attains this lower bound is necessarily minimal. There are,
however, only a couple of examples for which this lower bound is attained for
all $n$, most notable being the product Chebyshev weight functions on the square. 

In the present paper we shall show that the minimal cubature rules of degree 
$4n-1$ exist for a family of weight functions that includes, in particular,  
\begin{equation} \label{CWab}
 \CW_{\a,\b,\pm \frac12}(x,y) : = |x+y|^{2 \a+1} | x- y|^{2 \beta +1}
          (1-x^2)^{\pm \frac12}(1-y^2)^{\pm \frac12}, 
\end{equation}
on $[-1,1]^2$ and, furthermore, there is a connection between these minimal cubature
rules and Gaussian cubature rules associated with the weight function $W_{\a,\b,\pm \frac12}$.
The weight functions \eqref{CWab} include the product Chebyshev weight functions 
(when $\a = \b = \pm \frac12$), for which the minimal cubature rules are known to exist and
have been established in several different methods  \cite{BP, LSX, MP, X94}. Our result 
shows that they can be deduced from the Gaussian cubature rules for 
$W_{-\frac12, - \frac12, \pm \frac12}$ on $\Omega$. Giving the fact that so few minimal
cubature rules are known explicitly, this connection is rather surprising. 

Cubature rules are closely related to interpolation by polynomials. Based on the nodes of a 
Gaussian cubature rule of degree $2n-1$, there is a unique Lagrange interpolation 
polynomial of degree $n-1$ which converges to $f$ in $L^2$ norm as $n \to \infty$ (\cite{Xu92}). 
On the nodes of the minimal cubature rule that attains \eqref{lwbd}, there is a unique Lagrange
interpolation polynomial in an appropriate subspace of polynomials \cite{X94}. 
Furthermore, the interpolation polynomials based 
on the nodes of the minimal cubature rules for the produce Chebyshev weight function
$\CW_{-\frac12, -\frac12,-\frac12}$, studied in \cite{X96}, has the Lebesgue constant 
of order $(\log n)^2$ \cite{BMV}, which is the minimal order of projection operators 
on $[-1,1]^2$ \cite{SV}. We shall discuss the Lagrange interpolations based on both 
the nodes of Gaussian cubature rules with respect to $W_{\a,\b,\pm \frac12}$ and the 
nodes of minimal cubature rules for \eqref{CWab} in this paper. 

The paper is organized as follows. The next section is the preliminary, in which we 
recall basics on cubature rules and, in particular, the connection between cubature 
rules and interpolation polynomials, as well as basics on the orthogonal polynomials that
will be needed in the paper. The Gaussian cubature rules for weight functions
including $W_{\a,\b,\pm 12}$ and minimal cubature rules for $\CW_{\a,\b,\pm \frac12}$ 
are discussed in Sections 3 and 4, respectively. The interpolation polynomials based 
on the nodes of these cubature rules are treated in Sections 5 and 6, respectively. 

\section{Preliminary and Background}
\setcounter{equation}{0}

Minimal cubature rules are closely connected to orthogonal polynomials and to
polynomial interpolation. We recall the connections in this section and state 
necessary definitions and properties of the weight functions and their orthogonal 
polynomials that will be needed later in the paper. 

\subsection{Cubature, orthogonal polynomials and interpolation}
Let $W$ be a nonnegative weight function defined on a domain $\Omega$ in $\RR^2$ 
that has all finite moments; that is, $\int_\Omega x_1^j x_2^k W(x_1,x_2) dx_1dx_2 < \infty$ 
for all $j,k \in \NN_0$. Then  orthogonal polynomials of two variables with respect to $W$ 
exist. Let $\CV_n(W)$ denote the space of orthogonal polynomials of degree $n$ in two
variables. Then 
$$
   \dim \CV_n(W) = n+1. 
$$
Assume that $W$ is normalized so that $\int_\Omega W(x_1,x_2) dx_1dx_2 =1$. A basis
of $\CV_n(W)$, denoted by $\{P_{k,n}: 0 \le k \le n\}$, is mutually orthogonal if 
$$
  \int_\Omega P_{k,n}(x_1,x_2) P_{j,n}(x_1,x_2) W(x_1,x_2) dx_1d x_2 = 
              h_k \delta_{k,j}, \quad 0\le  k, j \le n,  
$$
where $h_k  > 0$ and it is called orthonormal if $h_k =1$ for $0 \le k \le n$. The 
reproducing kernel $K_n(W; \cdot,\cdot)$ of $\Pi_n^2$ in $L^2(W)$ is defined by 
$$
      \int_\Omega K_n(W;x,y)  p(y) W(y) dy = p(x), \quad \forall p \in \Pi_n^2,
$$
in which $x = (x_1,x_2)$ and $y=(y_1,y_2)$. If $P_{k,n}$ are orthonormal, then the reproducing 
kernel $K_n(W; \cdot,\cdot)$ of $\Pi_n^2$ in $L^2(W)$ is given by
\begin{equation} \label{eq:reprod}
 K_n(W; x, y) = \sum_{m=0}^n \sum_{k=0}^m P_{k,m}(x)P_{k,m}(y).
\end{equation}

Recall that a Gaussian cubature rule of degree $2n-1$, as in \eqref{cuba-generic}, has  
$\dim \Pi_{n-1}^2$ nodes. These nodes are necessarily common zeros of orthogonal 
polynomials in $\CV_n(W)$, that is, zeros of all polynomials in $\CV_n(W)$ (cf. \cite{DX, M, St}).

\begin{thm}\label{thm:Gauss}
Let $n \ge 1$. A Gaussian cubature rule of degree $2n-1$ exists if and only if its nodes
are common zeros of orthogonal polynomials of degree $n$.  Moreover, the weights $\l_k$
of the Gaussian cubature rule are given by 
$$
    \l_k = \left[ K_{n-1} (W; (x_k,y_k), (x_k,y_k)) \right ]^{-1}, \quad 1 \le k \le N.
$$
\end{thm}

Unlike interpolation in one variable, polynomial interpolation in two variables may not exist 
for a set of distinct points. It does exists if the interpolation points are nodes of a Gaussian
cubature rule \cite{X94}.

\begin{thm}\label{thm:Gauss-Interp}
Let $N=\dim \Pi_{n-1}^2$ and let $\{(x_k, y_k): 1 \le k \le N\}$ be the nodes of a Gaussian 
cubature rule of degree $2n-1$. Then there is a unique interpolation polynomial, $L_n f$, 
of degree $n-1$ that satisfies 
$$
   L_n f (x_{k,n},y_{k,n}) =  f(x_{k,n},y_{k,n}), \quad 1 \le k \le N.
$$
Furthermore, this interpolation polynomial is given explicitly by 
$$
    L_n f(x,y) = \sum_{k=1}^N f(x_k,y_k) \ell_k(x,y), \quad \ell_k(x,y):=\lambda_k K_{n-1}(W; (x,y), (x_k,y_k)).
$$
\end{thm}

For centrally symmetric weight functions, we consider the minimal cubature rules whose 
number of nodes attains the lower bound in \eqref{lwbd}. The nodes of such a cubature
rule are common zeros of a subspace of $\CV_n(W)$ (\cite{M}, see also \cite{X94}).  

\begin{thm}\label{thm:minimalCuba}
Let $n \ge 1$. The minimal cubature rule of degree $2n-1$ that attains the lower bound
\eqref{lwbd} exists if and only if its nodes are common zeros of $\lfloor \frac{n+1}{2} \rfloor +1$
orthogonal polynomials of degree $n$. 
\end{thm}

Since the number of nodes of such minimal cubature rules is $N= \dim \Pi_{n-1}^2 + 
\lfloor \frac{n}{2} \rfloor$, the polynomial that interpolates at the nodes of the cubature 
rule needs to be from a polynomial subspace of $\Pi_n$ that has dimension $N$. An obvious 
candidate of this subspace is the linear span of $\Pi_n^2 \setminus \CI_n$, where $\CI_n: = \{Q_{k,n}: k =0, 1,\ldots, 
\lfloor \frac{n+1}{2} \rfloor\}$ denotes a set of orthonormal polynomials that vanish on 
the nodes of the minimal cubature rule. Let $\{P_{k,n}: 1 \le k \le \lfloor \frac{n}{2} \rfloor\}$ 
be the orthonormal basis of the orthogonal complement of $\CI_n$ in $\CV_n(W)$. Then
$P_{k,n} \in \CV_n(W)$ and none of $P_{k,n}$ vanishes on all nodes of the cubature rule. 
We define a subspace $\Pi_n^*$ of $\Pi_n^2$ by 
\begin{equation}\label{Pin*}
      \Pi_n^*: = \Pi_{n-1}^2 \cup \mathrm{span} \left\{ P_{k,n}: 
                 1 \le k \le \left \lfloor \frac{n}{2} \right \rfloor \right\}. 
\end{equation}
The weights $\l_{k,n}$ of the minimal cubature rule in Theorem \ref{thm:minimalCuba} 
are given in the lemma below. 

\begin{lem}
Let $P_{k,n}$ be as in \eqref{Pin*}. There exists a sequence of positive numbers 
$\{b_{k,n}: 1 \le k \le \left \lfloor \frac{n}{2} \right \rfloor\}$, uniquely determined, such that 
the kernel $K_n^*(\cdot, \cdot)$ defined by 
\begin{equation} \label{Kn*}
 K_n^*(W; x,y) = K_{n-1}(W; x,y) + \sum_{k=1}^{\lfloor \frac{n}{2} \rfloor} b_{k,n} P_{k,n} (x) P_{k,n}(y),
\end{equation}
where $x = (x_1,x_2)$ and $y = (y_1,y_2)$, satisfies 
\begin{equation}\label{mcfWeight}
   \l_{k,n} =  \left[ K_n^* (W; (x_k,y_k), (x_k,y_k)) \right ]^{-1}, \quad 1 \le k \le N.
\end{equation}
\end{lem}

This lemma was proved in \cite{X94} and the coefficients were shown to be determined
by the matrix $[\C_n (P_{j,n} P_{k,n})]_{j,k = 0}^n$ in \cite{X97}, where $\{P_{0,n}, \ldots, P_{n,n}\}$ 
is an orthonormal basis of $\CV_n(W)$ and $\C_n f = \sum_{k=1}^N \l_k f(x_k,y_k)$ is the
minimal cubature rule. The kernel $K_n^*(\cdot, \cdot)$ can also be used for the Lagrange 
interpolation polynomials based on the nodes of the minimal cubature rules, as stated in
the following theorem \cite{X94}.

\begin{thm}\label{thm:minimal-Interp}
Let $W$ be a central symmetric weight function. Let $N=\dim \Pi_{n-1}^2 + \lfloor \frac{n}{2} \rfloor$
and let $\{(x_k, y_k): 1 \le k \le N\}$ be the nodes of the minimal cubature rule of degree $2n-1$. 
Then there is a unique interpolation polynomial, $\CL_n f$, in $\Pi_n^*$ that satisfies 
$$
        \CL_n f (x_{k,n},y_{k,n}) =  f(x_{k,n},y_{k,n}), \quad 1 \le k \le N.
$$
Furthermore, this interpolation polynomial is given explicitly by 
$$
   \CL_n f(x,y) = \sum_{k=1}^N f(x_k,y_k) \ell_k(x,y), \quad \ell_k(x,y):=\lambda_{k,n} K_n^*(W; (x,y), (x_k,y_k)),
$$
where $\l_{k,n}$ are the weights of the cubature rule given in \eqref{mcfWeight}.
\end{thm}

\subsection{Weight functions and orthogonal polynomials}
We define the weight functions for our Gaussian and minimal cubature rules. 
Our first weight function is defined on the domain 
$$
       \Omega: = \{(u,v): 1+u+v > 0, 1-u+v > 0, u^2 > 4v\}, 
$$
bounded by a parabola and two lines, as depicted in Figure 1. 
\begin{figure}[ht]
\includegraphics[scale=0.48]{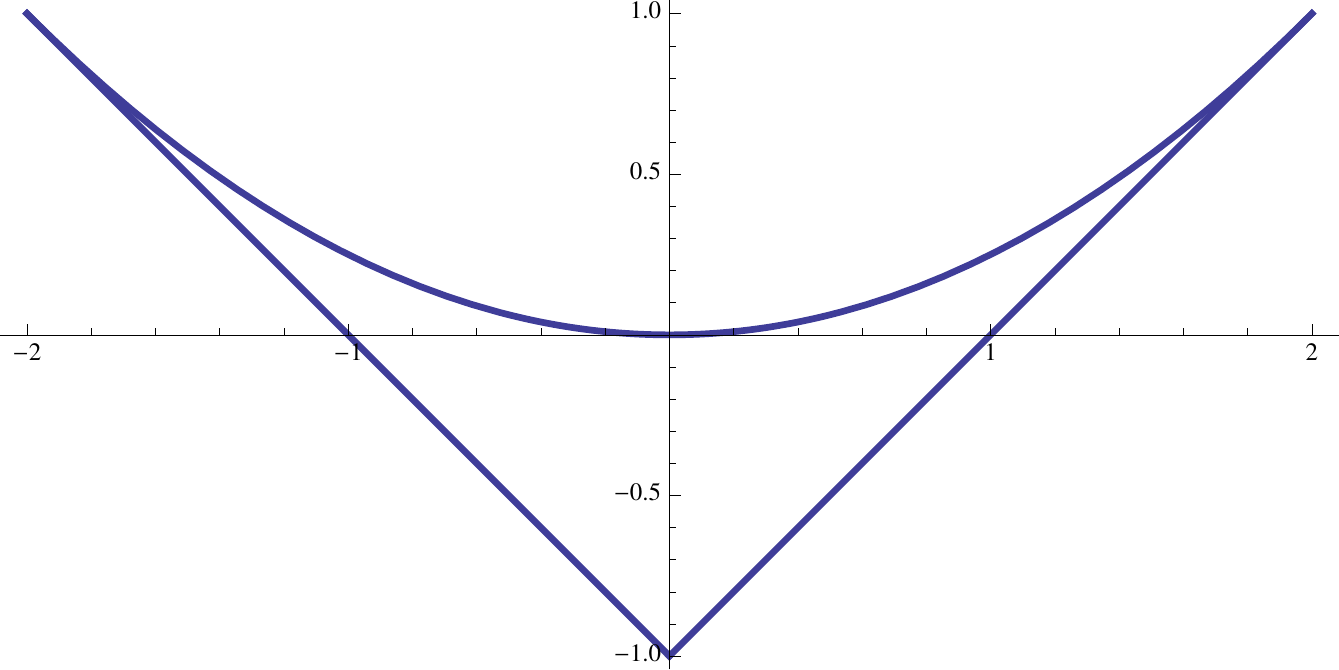}
\caption[Region of Koornwinder' orthogonal polynomials]{Domain $\Omega$}
\label{figure:region} 
\end{figure} 
Let $w$ be a nonnegative weight function defined on $[-1,1]$. We define 
\begin{equation} \label{Wgamma}
      W_\g(u,v)  := b_{w,\g} w(x) w(y) (u^2 - 4 v)^\g, \qquad (u,v) \in \Omega, 
\end{equation}
where the variables $(x,y)$ and $(u,v)$ are related by 
\begin{equation}   \label{u-v} 
        u = x+y, \quad v= xy
\end{equation}
and $b_{w,\g}$ is the normalization constant such that $\int_\Omega W_\g(u,v)dudv =1$.
In the case of the Jacobi weight function $w = w_{\a,\b}$ defined by 
$$
   w_{\a,\b}(x) := (1-x)^\a (1+x)^\b, \quad \a,\b > -1,
$$
the weight function $W_\gamma$ is denoted by $W_{\a,\b,\g}$ and it is given by 
\begin{equation} \label{Wabc}
W_{\a,\b,\g}(u,v) : = b_{\a,\b,\g}(1-u+v)^\a (1+u+v)^\b (u^2- 4 v)^\g, \quad (u,v) \in \Omega,
\end{equation}
where $\a, \b, \g > -1$, $\a + \g + \frac12> -1$ and $\b+ \g + \frac12> -1$ and \cite[Lemma 6.1]{S}, 
\begin{align}
 b_{\a,\b,\g}: = \frac{\sqrt{\pi} }{2^{2\a+2\b+4\g+2}} \frac{\Gamma(\a+\b+\g+\frac52)\Gamma(\a+\b+2 \g+3)}
    {\Gamma(\a+1) \Gamma(\b+1)\Gamma(\g+1) \Gamma(\a+\g + \frac32)
      \Gamma(\b+\g+\frac32)}.
\end{align}
 The weight function
$W_\g$ is related to $w(x)w(y)$ by the relation
\begin{align} \label{Para-Square}
  \int_{\Omega} f(u,v) W_\g(u,v) dudv  = 
       b_{w,\g} \int_{\triangle}  f(x+y, x y) w(x)w(y) |x-y|^{2\g+1}dxdy, 
\end{align}
where $\triangle : = \{(x,y): -1< x< y<1\}$. Since the integral in the right hand side has
symmetric integrand in $x, y$, it is equal to half of the integral over $[-1,1]^2$.
In particular, if $c_w$ is the normalization constant of $w$ on $[-1,1]$ so that 
$c_w \int_{-1}^1 w(x) dx =1$, then the normalization constant of $W_{-\frac12}$ 
is given by $b_{w, - \frac12} = 2 c_w^2$.

The orthogonal polynomials with respect to $W_{\a,\b,\g}$ were first studied by Koornwinder
in \cite{K74} and further studied in \cite{K75, KS, S}. They were applied to study cubature rules 
in \cite{SX}. In the case of $\g = \pm  \frac12$, 
the orthogonal polynomials with respect to $W_{\pm \frac12}$ can be given explicitly. Let 
$p_n$ denote the orthogonal polynomial of degree $n$ with respect to $w$.  
Then an orthonormal basis with respect to $W_{-\frac12}$ is given by 
\begin{equation} \label{OP-1/2}
    P_{k,n}^{(-\frac12)} (u,v) = \begin{cases} 
               p_n(x) p_k(y) + p_n(y) p_k(x), & 0 \le k < n, \\ 
               \sqrt{2} p_n(x) p_n(y), & k =n,\end{cases} 
\end{equation}
and an orthonormal basis with respect to $W_{\frac12}$ is given by 
\begin{equation} \label{OP+1/2}
    P_{k,n}^{(\frac12)} (u,v) = \frac{p_{n+1}(x) p_k(y) - p_{n+1} (y) p_k(x)}{x-y}, \quad 0 \le k \le n, 
\end{equation}
both families are defined under the mapping \eqref{u-v}. In the case of $W_{\a,\b,\g}$ we denote 
the orthogonal polynomials by $P_{k,n}^{\a,\b,\g}$. In particular, $P_{k,n}^{\a,\b,\pm \frac12}$
are expressible by the Jacobi polynomials. 

Our second family of weight functions are defined on $[-1,1]^2$ by 
\begin{equation} \label{CWgamma}
        \CW_\g (x,y) := W_\g (2 x y, x^2+y^2 -1)|x^2-y^2|, \qquad (x,y) \in [-1,1]^2, 
\end{equation}
where $W_\g$ is the weight function in \eqref{Wgamma}. In the case of $W_{\a,\b,\g}$, it becomes
\begin{align} \label{CWabc}
 \CW_{\a,\b, \g}(x,y) : = b_{\a,\b,\g}4^\g |x-y|^{2\a +1}  |x + y|^{2\b +1}   (1-x^2)^\g (1-y^2)^\g, 
\end{align}
which includes \eqref{CWab}. 
The $\CW_\g$ is normalized if $W_\g$ is because of the integral relation
\begin{align} \label{Int-P-Q}
 \int_{\Omega} f(u,v) W_{\g} (u,v) du dv  = \int_{[-1,1]^2} f(2xy, x^2+y^2 -1) \CW_{\g}(x,y) dx dy.
\end{align}
The orthogonal polynomials with respect to $\CW_{\g}$ can be expressed in terms of
orthogonal polynomials with respect to $W_\g$ (\cite{X10}). For this paper we will only need a 
basis for $\CV_{2n}(W_\g)$, which consists of polynomials
\begin{align} \label{Qeven}
\begin{split}
   {}_1Q_{k,2n}^{(\g)}(x,y):= & P_{k,n}^{(\g)}(2xy, x^2+y^2 -1), \qquad 0 \le k \le n, \\
   {}_2Q_{k,2n}^{(\g)}(x,y) := & b_\g^{(1,1)}(x^2-y^2)  P_{k,n-1}^{(\g),1,1}(2xy, x^2+y^2 -1),  
   \quad 0 \le k \le n-1, 
 \end{split}
\end{align}
where $P_{k,n-1}^{(\g),1,1}$ are orthonormal polynomials with respect to the weight function
$(1-u+v)(1+u+v)W_\g(u,v)$ and $b_\g^{(1,1)}$ is a normalization constant for the weight 
function. In the case of $\CW_{\a,\b,\g}$, we denote the orthogonal polynomials by 
${}_iQ_{k,2n}^{\a,\b,\g}$, in which case, $P_{k,n-1}^{(\g), 1,1} = P_{k,n-1}^{\a+1,\b+1,\g}$ in 
\eqref{Qeven}. We will need explicit formulas for these polynomials when $\g = -1/2$, which 
we sum up in the following subsection.  Further results on orthogonal polynomials with respect to $W_\g$ 
can be found in \cite{X10}. 

\subsection{Jacobi polynomials and orthogonal polynomials for $\CW_{\a,\b,-\frac12}$}
 
The Jacobi polynomials are orthogonal with respect to $w_{\a,\b}$ and they 
are given explicitly by a hypergeometric function as 
$$
 P_n^{(\a,\b)}(x,y) = \frac{(\a+1)_n}{n!} {}_2F_1 \left(\begin{matrix} -n, n+\a+\b+1 \\ \a+1 \end{matrix};
        \frac{1-x}{2} \right) = l_n^{(\a,\b)} x^n+ \ldots, 
$$
where $l_n^{(\a,\b)}$ is the leading coefficient. By \cite[(4.21.6)]{Sz}, 
\begin{align} \label{lead-cn}
    l_n^{(\a,\b)} = \frac{(n + a + b + 1)_n}{2^n n!} \quad \hbox{and} \quad
       c_{\a,\b}: =  \frac{\Gamma(\a+\b+1)}{2^{\a+\b+1} \Gamma(\a+1)\Gamma(\b+1)}. 
\end{align}
The Jacobi polynomials satisfy the orthogonality conditions
\begin{align*}
  c_{\a,\b}  \int_{-1}^1 P_n^{(\a,\b)} (x) P_m^{(\a,\b)} (x) w_{\a,\b}(x) dx 
      =  h_n^{(\a,\b)} \delta_{n,m}, 
\end{align*}
where 
\begin{align} \label{hnab}
       h_n^{(\a,\b)}:=  \frac{(\a+1)_n (\b+1)_n (\a+\b+n+1)}{n!(\a+ \b+2)_n (\a+\b+ 2n+1)}.
\end{align}
The reproducing kernel of $k_n^{(\a,\b)}$ of the space of polynomials of degree at most $n$ 
is given by, according to the Christoffel-Darboux formula, 
\begin{align} \label{Jacobi-kernel}
   k_n^{(\a,\b)}(x,y) = &  \frac{2 (n + 1)!  (\a + \b + 2)_n}{(2 n + \a + \b + 2) (\a + 1)_n (\b + 1)_n}  \\
           &  \times   \frac{P_{n+1}^{(\a,\b)}(x)P_{n}^{(\a,\b)}(y)-P_{n+1}^{(\a,\b)}(y)P_{n}^{(\a,\b)}(x)}{x-y}. \notag
\end{align}
The Gaussian quadrature of degree $2n-1$ for the Jacobi weight is given by 
\begin{equation} \label{Gauss-quadrature}
 c_{\a,\b} \int_{-1}^1 f(x) w_{\a,\b}(x)dx  = \sum_{k=1}^n  \l_n^{(\a,\b)} f(x_{k,n}), \quad \forall f \in \Pi_{2n-1},
\end{equation}
where $x_{1,n}, \ldots, x_{n,n}$ are the zeros of the Jacobi polynomial $P_n^{(\a,\b)}$ and 
$$
     \lambda_n^{(\a,\b)} = [k_n^{(\a,\b)}(x_{k,n},x_{k,n})]^{-1}. 
$$
We denote the orthonormal Jacobi polynomials by $p_n^{(\a,\b)}$. It follows readily that 
$p_n^{(\a,\b)}(x) = (h_n^{(\a,\b)} )^{-\frac12} P_n^{(\a,\b)}(x)$. The following lemma will be needed
in Section 6. 

\begin{lem} \label{hat-hn}
For $\a,\b > -1$ and $m \ge  0$, define
\begin{align}\label{hathn}
 \wh h_m :=  \sum_{k=1}^{n} \l_k^{(\a,\b)} (1-x_k^2)\left[ p_{n-1}^{(\a+1,\b+1)} (x_k) \right ]^2. 
\end{align}
Then $\wh h_m = c_{\a,\b}/c_{\a+1,\b+1}$ for $0 \le m \le n-2$, and 
$$
     \wh h_{n-1} =  \frac{4 (1 + \a) (1 + \b) (1 + \a + \b + 2 n)}{(2 + \a + \b) (3 +\a + \b) (1 + \a + \b + n)}. 
$$
\end{lem}

\begin{proof}
For $0\le m \le n-2$, we can apply Gaussian quadrature and use the orthonormality of $p_m^{(\a+1,\b+1)}$ to
conclude, since $(1-x^2)w_{\a,\b}(x) = w_{\a+1,\b+1}(x)$,  
$$
  \wh h_m = c_{\a,\b} \int_{-1}^1 (1-x^2)  \left[p_m^{(\a+1,\b+1)}(x) \right]^2 w_{\a,\b}(x)dx =  \frac{c_{\a,\b}}{c_{\a+1,\b+1}}.
$$
For $m = n-1$, we cannot apply the Gaussian quadrature of degree $2n-1$ directly, since $(1-x^2)  
[p_m^{(\a+1,\b+1)}(x)]^2$ has degree $2n$. However, by \cite[(4.5.5)]{Sz}, 
\begin{align*}
   (1-x_k^2) P_{n-1}^{(\a+1,\b+1)}(x_k) = & \frac{4(n+\a)(n+\b)}{2n+\a+\b)(2n+\a+\b+1)} P_{n-1}^{(\a,\b)}(x_k) \\
      & - \frac{4n (n+1)}{2n+\a+\b+1)(2n+\a+\b+2)} P_{n+1}^{(\a,\b)}(x_k),
\end{align*}
and by the three-term relation satisfied by $\{P_n^{(\a,\b)} \}_{n \ge 0}$ \cite[(4.5.1)]{Sz}, 
$$
      (n+1)(n+\a+\b+1)(2n+\a+\b) P_{n+1}^{(\a,\b)}(x_k) 
         = -  (n+\a)(n+\b)(2n+\a+\b+2) P_{n-1}^{(\a,\b)}(x_k). 
$$
From these two equations it follows that 
\begin{equation}\label{Jacobi-property}
      (1-x_k^2) P_{n-1}^{(\a+1,\b+1)}(x_k) =  \frac{4(n+\a)(n+\b)}{(2n+\a+\b)(n+\a+\b+1)}P_{n-1}^{(\a,\b)}(x_k). 
\end{equation}
Denote the coefficient in front of $P_{n-1}^{(\a,\b)}(x_k)$ in the above equation by $D_n$. Then, by the 
Gaussian quadrature and the orthogonality of the Jacobi polynomials, 
\begin{align*}
  \wh h_{n-1} & = D_n \left[h_{n-1}^{(\a+1,\b+1)}\right]^{-1} \sum_{k=1}^{n} \l_k^{(\a,\b)} P_{n-1}^{(\a+1,\b+1)} (x_k)  P_{n-1}^{(\a,\b)}(x_k) \\
        & = D_n \left[h_{n-1}^{(\a+1,\b+1)}\right]^{-1} c_{\a,\b}  \int_{-1}^1  P_{n-1}^{(\a+1,\b+1)} (x)  P_{n-1}^{(\a,\b)}(x) w_{\a,\b}(x)dx \\
        & = D_n \left[h_{n-1}^{(\a+1,\b+1)}\right]^{-1} c_{\a,\b} \frac{l_{n-1}^{(\a+1,\b+1)}}{l_{n-1}^{\a,\b}}
                  \int_{-1}^1  \left[ P_{n-1}^{(\a,\b)} (x) \right]^2 w_{\a,\b}(x)dx \\
        & = D_n \frac{h_{n-1}^{(\a,\b)}} {h_{n-1}^{(\a+1,\b+1)}} \frac{l_{n-1}^{(\a+1,\b+1)}}{l_{n-1}^{\a,\b}},
\end{align*}
which simplifies, by \eqref{lead-cn} and \eqref{hnab}, to the stated result for $\wh h_{n-1}$.
\end{proof}

In Section 6 we will need the explicit formula of orthogonal polynomials and reproducing 
kernels for the weight function $\CW_{\a,\b,-1/2}$, which we rename as 
\begin{equation} \label{CWabab}
   \CW_{\a,\b}(x,y): =    \CW_{\a,\b,-1/2}(x,y) = 2 c_{\a,\b}^2 \frac{|x-y|^{2 \a+1}|x+y|^{2\b+1}}{ \sqrt{1-x^2} \sqrt{1-y^2}}.
\end{equation} 
In this case the orthonormal polynomials of even degree in \eqref{Qeven} are given in terms of Jacobi
polynomials, which are  

\begin{prop} \label{Q2nJacobi}
Let $\a, \b > -1$. An orthonormal basis of $\CV_{2n}(\CW_{\a,\b,-\frac12})$ is given by, for $0 \le k \le n$ 
and $0 \le k \le n-1$, respectively,  
\begin{align*} 
  & {}_1Q_{k,2n}^{(\a,\b)}(\cos\t,\cos \phi) =  p_n^{(\a,\b)} (\cos (\t - \phi)) p_k^{(\a,\b)}(\cos (\t+\phi)) \\
  &    \qquad\qquad  \qquad\qquad \qquad\qquad  \qquad 
             +  p_k^{(\a,\b)} (\cos (\t - \phi)) p_n^{(\a,\b)}(\cos (\t+\phi)), \\
  & {}_2Q_{k,2n}^{(\a,\b)}(\cos\t,\cos \phi)  = 
        \gamma_{\a,\b} (x^2-y^2) \left[ p_{n-1}^{(\a+1,\b+1)} (\cos (\t - \phi)) p_k^{(\a+1,\b+1)} (\cos (\t+\phi)) \right. \\ 
  &    \qquad\qquad  \qquad\qquad \qquad\qquad  \qquad 
        \left. +  p_k^{(\a+1,\b+1)} (\cos (\t - \phi)) p_{n-1}^{(\a+1,\b+1)} (\cos (\t+\phi)) \right],
\end{align*}
where ${}_1Q_{n,2n}^{(\a,\b)}$ and ${}_2Q_{n,2n}^{(\a,\b)}$ are multiplied by $\sqrt{2}/2$ 
and $\gamma_{\a,\b} = c_{\a+1,\b+1}/(\sqrt{2} c_{\a,\b})$. 
\end{prop}
Denote the reproducing kernel of $\Pi_n^2$ with respect to $\CW_{\a,\b}$ by 
$\CK_n^{\a, \b}(\cdot,\cdot)$. By \cite[Theorem 4.8]{X10}, the kernel $\CK_{2n-1}^{\a,\b}$ 
is given explicitly by 
\begin{align} \label{eq:reprodCW}
    \CK_{2n-1}^{\a,\b}(x,y)  =  K_{n-1}^{\a,\b}(s,t) &\, +  
        d_{\a,\b}^{(1,1)} (x_1^2-x_2^2)(y_1^2-y_2^2) K_{n-2}^{\a+1,\b+1}(s,t) \\
            &\, + d_{\a,\b}^{(0,1)}   (x_1+x_2)(y_1+y_2) K_{n-1}^{\a,\b+1}(s,t) \notag\\
            &\, + d_{\a,\b}^{(1,0)}  (x_1-x_2)(y_1-y_2) K_{n-1}^{\a+1,\b}(s,t),  \notag
\end{align}
where $s = (2 x_1x_2, x_1^2+x_2^2-1)$, $t = (2 y_1 y_2, y_1^2+y_2^2 -1)$, 
$d_{\a,\b}^{(i,j)} = c_{\a+i,b+j}^2 / c_{\a,\b}^2$ and, with $(x_1,x_2) =(\cos \t, \cos \t_2)$ and 
$(y_1,y_2) = (\cos \phi_1, \cos \phi_2)$,  
\begin{align} \label{Kn-st-1/2}
      K_n^{\a,\b}(s,t)   
        :=  & \frac12\left[ k_n^{\a,\b}(\cos (\t_1-\t_2), \cos (\phi_1 - \phi_2))  
            k_n^{\a,\b}(\cos (\t_1+\t_2), \cos (\phi_1 + \phi_2))   \right.   \\
       & \left. +  k_n^{\a,\b}(\cos (\t_1-\t_2), \cos (\phi_1 + \phi_2)) 
               k_n^{\a,\b}(\cos (\t_1+\t_2), \cos (\phi_1 - \phi_2)) \right]. \notag 
\end{align}
For orthonormal basis of odd degrees and the reproducing kernels of even degrees, as well
as other results on them, see \cite{X10}. 

\section{Gaussian cubature rules}
\setcounter{equation}{0}

%In this section we consider cubature rules on several domains. 
%Our starting point is the Gaussian cubature rules for $W_{\pm \frac12}$ on $\Omega$. 
%For completeness, we discuss the construction of such rules in the first subsection. 
%\subsection{Gaussian cubature on $\Omega$} 
In this section we consider Gaussian cubature rules for $W_{\pm \frac12}$ on $\Omega$
and their transformations. We shall show that these rules can be transformed into minimal 
cubature rules for $\CW_{\a,\b,-1/2}$ on $[-1,1]^2$ in the next section. The first proof 
that Gaussian cubature rules exist for $W_{\pm \frac12}$  was given in \cite{SX} via the 
structure matrices of orthogonal polynomials. Below is another proof that is of independent 
interest. 

We start with the Gaussian quadrature rule for the integral against $w$ on $[-1,1]$, 
\begin{equation} \label{Gauss[-1,1]}
  c_w \int_{-1}^1 f(x) w(x) dx = \sum_{k=1}^n \l_k f(x_{k,n}), \quad f \in \Pi_{2n-1},
\end{equation}
where $\Pi_{2n-1}$ denotes the space of polynomials of degree $2n-1$ in one variable
and $c_w$ is the normalization constant so that $c_w \int_{-1}^1 w(x) dx =1$. It is known 
that $\l_k > 0$ and $x_{k,n}$ are zeros of the orthogonal polynomial $p_n$ with respect 
to $w$. When $w = w_{\a,\b}$, the orthogonal polynomials are the Jacobi polynomials 
$P_n^{(\a,\b)}$ and $x_{k,n}$, $1 \le k \le n$, are the zeros of $P_n^{(\a,\b)}$. 
We define 
\begin{equation}\label{ujk}
  u_{j,k} = u_{j,k,n}: = x_{j,n}+ x_{k,n} \quad\hbox{and}\quad  v_{j,k} = v_{j,k,n}:=x_{j,n}x_{k,n}. 
\end{equation}

\begin{thm} \label{thm:Gaussian}
For $W_{- \frac12}$ on $\Omega$, the Gaussian cubature rule of degree $2n-1$ is 
\begin{equation} \label{GaussCuba-}
  \int_{\Omega} f(u, v) W_{-\frac12}(u,v) du dv =2 \sum_{k=1}^n \mathop{ {\sum}' }_{j=1}^k  
          \l_k \l_j f(u_{j,k}, v_{j,k}), \quad f \in \Pi_{2n-1}^2, 
\end{equation}
where ${\sum}'$ means that the term for $j = k$ is divided by 2. For $W_{\frac12}$ on 
$\Omega$, the Gaussian cubature rule of degree $2n-3$ is 
\begin{equation} \label{GaussCuba+}
   \int_{\Omega} f(u, v) W_{\frac12}(u,v) du dv = 2 \sum_{k=2}^n \sum_{j=1}^{k-1}  
          \l_{j,k} f(u_{j,k}, v_{j,k}), \quad f \in \Pi_{2n-3}^2, 
\end{equation} 
where $\lambda_{j,k} = \l_j \l_k (x_{j,n} - x_{k,n})^2$. 
\end{thm}

\begin{proof}
The product of \eqref{Gauss[-1,1]} is a cubature rule on $[-1,1]^2$
\begin{equation} \label{ProdGauss}
    c_w^2 \int_{[-1,1]^2} f(x ,y) w(x)w(y) dx dy = \sum_{k=1}^n \sum_{j=1}^n 
          \l_k \l_j f(x_{k,n}, x_{j,n}), 
\end{equation}
which is exact for $f \in \Pi_{2n-1} \times \Pi_{2n-1}$, the space of polynomials of degree 
at most $2n-1$ in either $x$ or $y$ variable. Applying \eqref{ProdGauss} on the symmetric 
polynomials $f(x+y, xy)$ and using the symmetry, we obtain 
$$
  c_w^2 \int_{\triangle} f(x+y, xy) w(x)w(y) dx dy = \sum_{k=1}^n \mathop{ {\sum}' }_{j=1}^k  
          \l_k \l_j f(x_{k,n}+ x_{j,n}, x_{k,n}x_{j,n}),  
$$
exact for polynomials $f$ in $\Pi_{2n-1}\times \Pi_{2n-1}$.  Under the change of variables
$u = x+y$ and $v = xy$ and by \eqref{Para-Square}, the above cubature becomes 
\eqref{GaussCuba-}, since $\Pi_{2n-1}\times \Pi_{2n-1}$ becomes $\Pi_{2n-1}^2$ under 
the mapping $(x,y) \mapsto (u,v)$. It is easy to see that \eqref{GaussCuba-} has 
$\dim \Pi_{n-1}^2$ nodes, so that it is a Gaussian cubature rule. 

To prove \eqref{GaussCuba+}, we apply the product Gaussian cubature rule 
\eqref{ProdGauss} on functions of the form $(x - y)^2 f(x+y,xy)$ for $f \in \Pi_{2n-2} 
\times \Pi_{2n-3}$ to get 
$$
  \int_{\triangle} f(x+y, xy) (x-y)^2 w(x)w(y) dx dy = \sum_{k=2}^n \sum_{j=1}^{k-1}  
          \l_j \l_k (x_{j,n} - x_{k,n})^2 f(u_{j,k}, v_{j,k}).  
$$
Since $(x-y)^2 w(x)w(y) = W_{\frac12}(u,v)$ for $u = x+y$ and $v = xy$, the above 
cubature rule becomes \eqref{GaussCuba+} under $(x,y) \mapsto (u,v)$.
\end{proof}

By Theorem \ref{thm:Gauss}, the nodes $\{(x_{k,n}+x_{j,n}, x_{k,n}x_{j,n}): 1 \le j \le k \le n\}$ 
of the cubature rule \eqref{GaussCuba-} are common zeros of the orthogonal polynomials 
in $\{P_{0,n}^{(-\frac12)}, \ldots, P_{n,n}^{(-\frac12)}\}$, and the nodes 
$\{(x_{k,n}+x_{j,n}, x_{k,n}x_{j,n}): 1 \le j \le k \le n-1\}$ of the cubature rules 
\eqref{GaussCuba+} are common zeros 
of $\{P_{0,n-1}^{(\frac12)}, \ldots,  P_{n-1,n-1}^{(\frac12)}\}$. Formulated in the
language of algebraic geometry, this states, for example, that the polynomial ideal 
$I = \la P_{k,n}^{(-\frac12)}, \ldots,  P_{n,n}^{(-\frac12)}\ra$ has the zero-dimensional 
variety $V= \{(x_{k,n}+x_{j,n}, x_{k,n}x_{j,n}): 1 \le j \le k \le n\}$. 

We remark that the above procedure of deriving cubature rules for $W_{\pm \frac12}$ 
on $[-1,1]^2$ can be adopted for other types of cubature rules besides Gaussian cubature
rules. In fact, instead of starting with the product Gaussian cubature rules for $w(x) w(y)$ 
on $[-1,1]^2$ as in the proof of Theorem \ref{thm:Gaussian}, we can start with a 
product cubature rule of other types. For example, we can start with a quadrature rule 
of degree $2n$ for $w$ that has all nodes inside $[-1,1]$, in which case an analogue of 
Theorem \ref{thm:Gaussian} was established in \cite{SX}. We can also start with a 
Gauss-Lobatto quadrature for $w$ to get a cubature rule that has nodes also on the two 
linear branches of the boundary of $\Omega$.  

The Theorem \ref{thm:Gaussian} shows that Gaussian cubature rules exist for 
$W_{\pm \frac12}$. An immediate question is if Gaussian cubature rules 
also exist for the weight function $W_\g$ for $\g \ne \pm \frac12$. The answer, however, 
is negative. 

\begin{thm} \label{NoGaussian}
For $n \ge 1$, the Gaussian cubature rules do not exist for $W_{-1/2, -1/2, \g}$ if 
$\g \ne \pm 1/2$. 
\end{thm}

\begin{proof}
It was shown in \cite[(10.7)]{S} that a basis of orthogonal polynomials of degree $n$ 
with respect to $W_{-1/2, -1/2, \g}$ is given explicitly by
$$
   P_{k,n}^{-\frac12, -\frac12,\g}(2 x y, x^2+y^2 -1) = 
      P_{n+k}^{(\g,\g)}(x)P_{n-k}^{(\g,\g)}(y) + P_{n-k}^{(\g,\g)}(x)P_{n+k}^{(\g,\g)}(y), , 
$$ 
where $\, 0 \le k \le n$ and $P_n^{(\a,\b)}$ is the Jacobi polynomial of degree $n$. 
It is easy to see that these polynomials do not have common zeros (considering, 
for  example, $k = n$ first). Consequently, $\{P_{k,n}^{-\frac12,-\frac12, - \g} (x,y): 
0 \le k \le n\}$ does not have $\dim \Pi_{n-1}^2$ common zeros for $n \ge 1$. 
Hence, the Gaussian cubature rules do not exist according to Theorem \ref{thm:Gaussian}. 
\end{proof}

For what we will do in the following subsection, we make an affine change of variables 
$u = 2 (s - t)$ and $v = 2s  + 2 t -1$, which implies that the measure becomes 
$W_\g(u, v) dudv =W_\g^*(s,t) dsdt$, where 
\begin{align} \label{W*g}
  W_\g^*(s,t) :=& 2 b_w^\g 4^{\g+1}  w(x)w(y)  ( (1- \sqrt{s})^2 -t )^\g( (1+ \sqrt{s})^2 -t )^\g \\
      & \hbox{with} \,\, s = \tfrac14 (1+x)(1+y), \, t = \tfrac14 (1-x)(1-y),  \notag
\end{align}
and the domain $\Omega$ becomes $\Omega^*$ defined by
$$
       \Omega^* := \{(s,t): s \ge 0, \,\, t \ge 0, \,\, \sqrt{s}+\sqrt{t} \le 1\}, 
$$
which is depicted in the right figure of Figure 2.  
\begin{figure} [ht]
\includegraphics[scale=0.55]{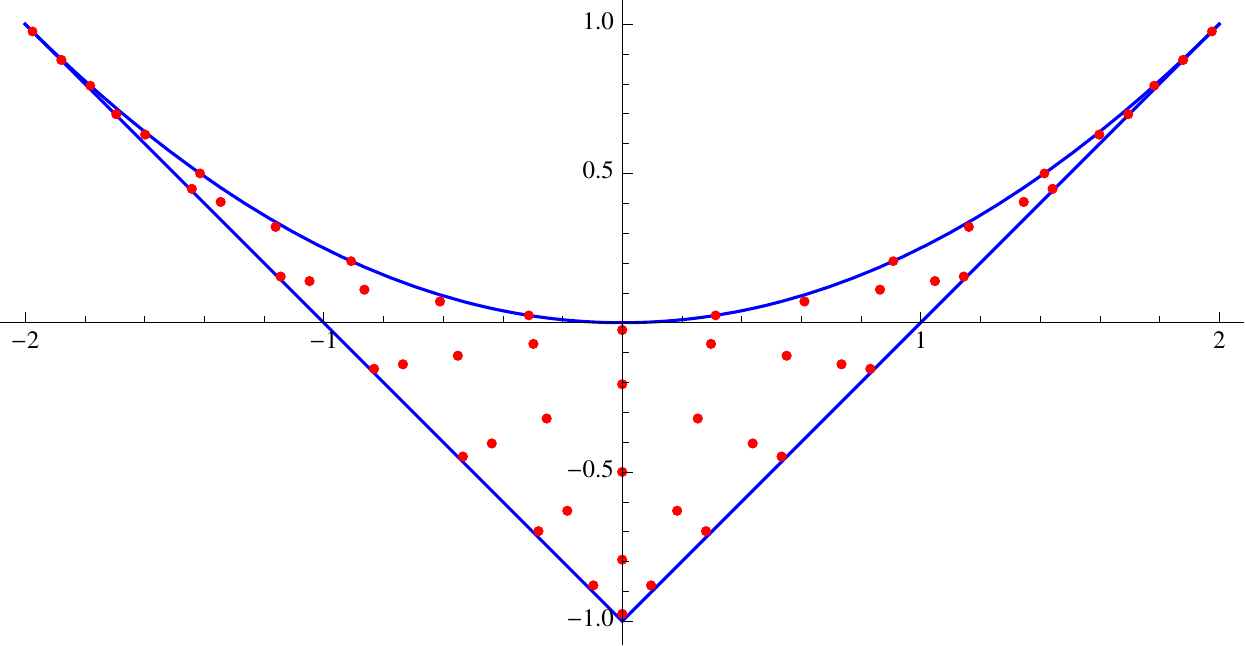} \quad \includegraphics[scale=0.43]{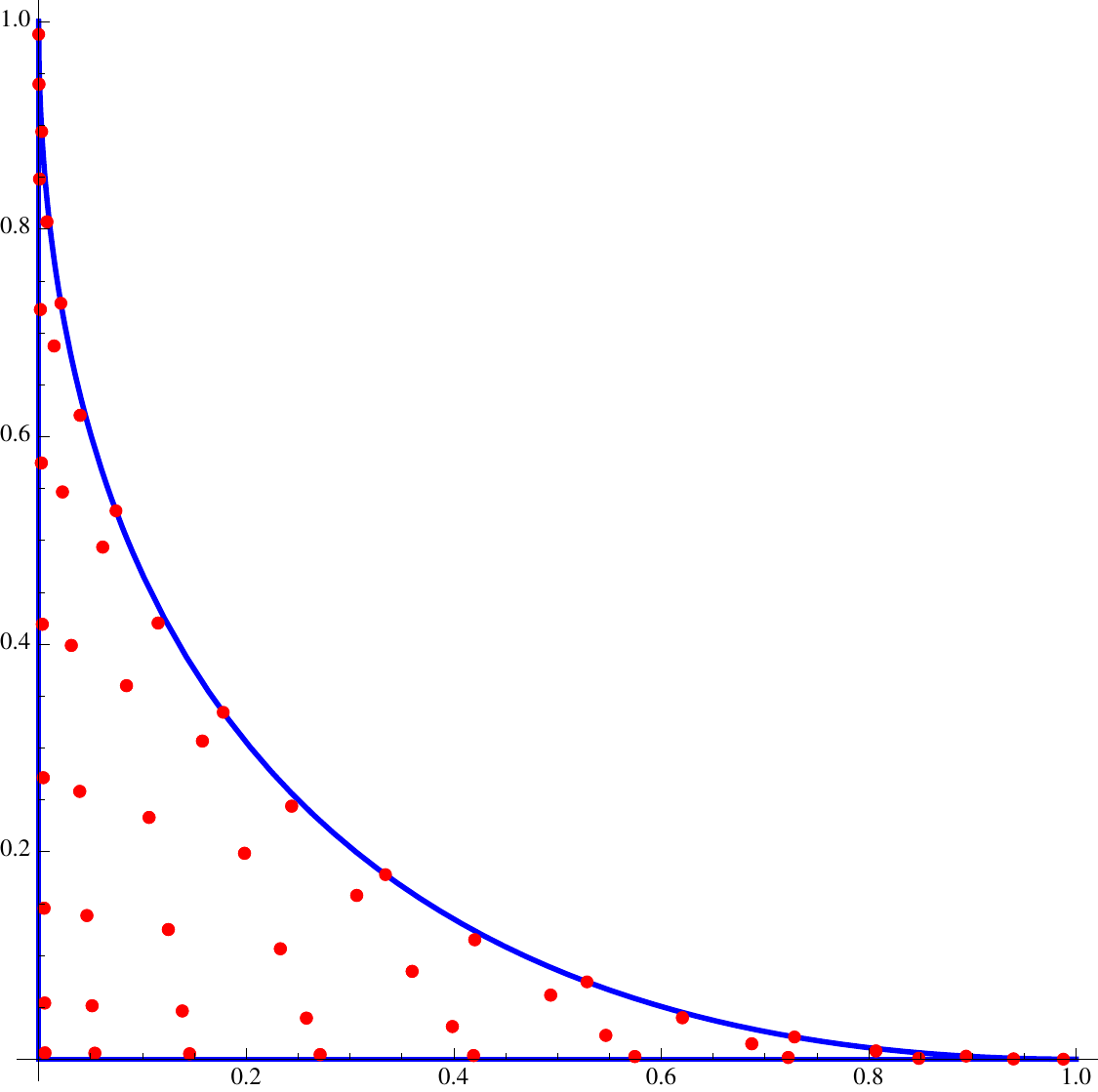}   
\caption{Nodes of cubature rules of degree 19 for $W_{-\frac12,-\frac12, -\frac12}$ 
and $W^*_{-\frac12,-\frac12, -\frac12}$}
\label{figure:region2} 
\end{figure} 
In the case of $w = w_{\a,\b}$ the weight 
function $W^*_\g$ becomes 
$$
  W^*_{\a,\b,\g}(s,t) =2  b_{\a,\b,\g} 4^{\a+\b+\g+1} s^\a t^\b  ( (1- \sqrt{s})^2 -t )^\g( (1+ \sqrt{s})^2 -t )^\g.
$$
Since the affine transform does not change the strength of the cubature rules, the Gaussian cubature 
rules exist for
the weight functions $W^*_{\pm \frac12}$. Let us denote by
$$
    x_{j,k}^* = x_{j,k,n}^* :=  \tfrac14 (1+x_{j,n})(1+x_{k,n}), \quad
        y_{j,k}^* = y_{j,k,n}^* :=  \tfrac14 (1- x_{j,n})(1 - x_{k,n}).
$$

\begin{cor} \label{cubaW*}
For $W^*_{- \frac12}$ on $\Omega^*$, the Gaussian cubature rule of degree $2n-1$ is 
\begin{equation} \label{GaussCuba2-}
     \int_{\Omega^*} f(s, t) W^*_{-\frac12}(s,t) dsdt = 2 \sum_{k=1}^n \mathop{ {\sum}' }_{j=1}^k  
          \l_k \l_j f(x_{j,k}^*, y_{j,k}^*), \quad f \in \Pi_{2n-1}^2. 
\end{equation}
For $W_{\frac12}^*$ on $\Omega^*$, the Gaussian cubature rule of degree $2n-3$ is 
\begin{equation} \label{GaussCuba2+}
   \int_{\Omega^*} f(s, t) W_{\frac12}^*(s,t) ds dt = 2 \sum_{k=2}^n \sum_{j=1}^{k-1}  
          \l_{j,k} f(x_{j,k}^*, y_{j,k}^*), \quad f \in \Pi_{2n-3}^2. 
\end{equation} 
\end{cor}

The nodes of the cubature ruled of degree 19 for $W_{-\frac12,-\frac12, -\frac12}$ 
and $W^*_{-\frac12,-\frac12, -\frac12}$ are depicted in the left and the right figures 
of Figure 2, respectively.

\section{Minimal cubature rules}
\setcounter{equation}{0}

Our goal in this section is to establish minimal cubature rules for the weight 
functions $\CW_\g$ on $[-1,1]^2$. We shall do so by several transformations 
of the Gaussian cubature rules in the previous section. 

First we recall the Sobolev theorem on invariant cubature rules. A cubature rule in the form 
of \eqref{cuba-generic} is invariant under a finite group $G$ if the equality is unchanged 
under $f \mapsto  \sigma f$, where $\sigma f(x) = f (x \sigma)$,  for all $\s \in G$. The
Sobolev theorem states that if a cubature is invariant under $G$ then it is exact for a 
subspace $\CP$ of polynomials if and only if it is exact for all polynomials in $\CP$ that
are invariant under $G$. 

We start from cubature rules for $W_\g^*$ in Corollary \ref{cubaW*} and make a change of 
variables $(s,t) \mapsto (u^2, v^2)$. The domain $\Omega^*$ becomes the triangle 
$T:= \{(u,v): u, v \ge 0, 1-u-v \ge 0\}$ and, since the weight function $W_\g^*(u^2,v^2)$
is even in both $u$ and $v$, we extend it by symmetry to the rhombus $R$, depicted in 
Figure 3, 
$$
   R : = \{ (u,v):  -1 < u + v < 1, \, \, -1 < u-v < 1\}. 
$$
\begin{figure} [ht]
\includegraphics[scale=0.65]{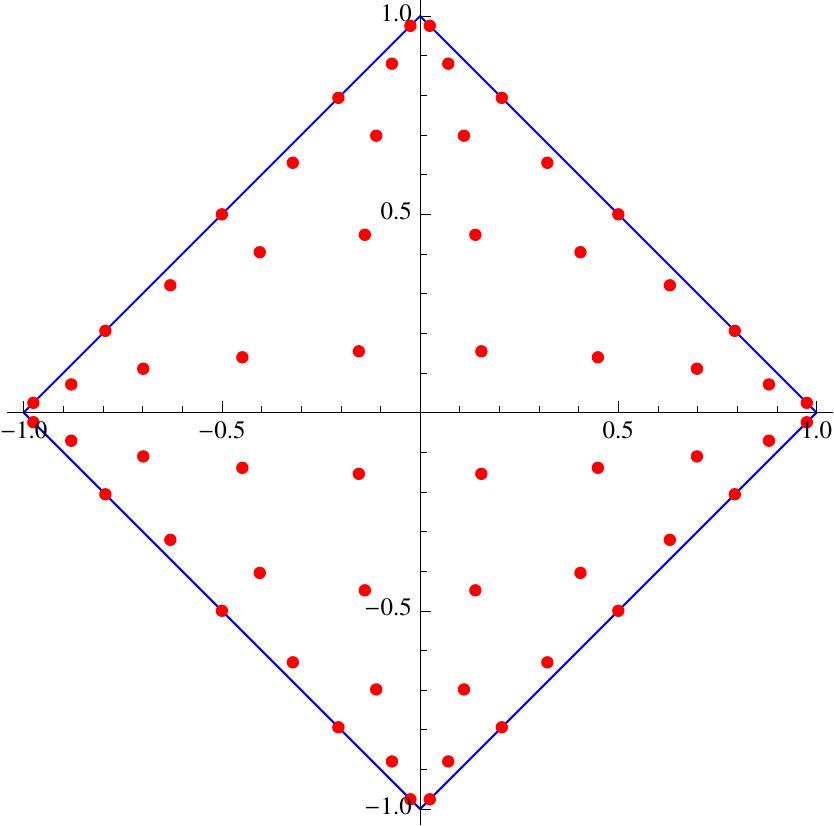}
\caption{Nodes of the minimal cubature rule of degree 19 for $U_{-\frac12,-\frac12,-\frac12}$}
\label{figure:region3} 
\end{figure} 
The change of variables has a Jacobian $ds dt = 4 |u v| dudv$. We define the weight 
function on $R$ by
\begin{align*}
   U_\g(u,v):= & |u v| W_\g^*(u^2,v^2) = 
       2 b_w^\g 4^{\g+1} w(x)w(y) |uv| ( (1- u)^2 -v^2 )^\g( (1+ u)^2 - v^2)^\g, \\
       &  \quad \hbox{where}\quad u = \tfrac12 \sqrt{1+x}\sqrt{1+y}, \,\, v = \tfrac12 \sqrt{1-x}\sqrt{1-y}. 
\end{align*}
In the case of $W_{\a,\b,\g}^*$, the corresponding weight function is 
\begin{equation}\label{U*}
  U_{\a,\b,\g}(u,v) =  2 b_{\a,\b,\g}4^{\a+\b+\g+1} u^{2\a+1} v^{2\b+1}  ( (1- u)^2 -v^2 )^\g( (1+ u)^2 - v^2)^\g.
\end{equation}
Under the change of variables $(s,t) \mapsto (u,v)$ and using the symmetry, the integrals 
are related by 
\begin{equation}\label{IntR}
  \int_{\Omega^*} f(s, t) ds dt = 4 \int_T f(u^2, v^2) u v dudv = \int_R f(u^2, v^2)|u v| dudv,
\end{equation}
from which it is easy to see that $U_\g$ satisfies $\int_R U_\g(u,v) dudv =1$. 

Directly from its definition, the weight function $U_\g$ is evidently centrally symmetric. 
To state the cubature rules for $U_\g$, we introduce the notation $\t_{k,n}$ by
$$
    x_{k,n} = \cos \t_{k,n}, \qquad k =0,1,\ldots, n.
$$
Since $w$ is supported on $[-1,1]$, the zeros of the orthogonal polynomial $p_n$ are all
inside $[-1,1]$, so that $0 < \t_{k,n} < \pi$.  

\begin{thm}
For $U_{- \frac12}$ on the rhombus $R$, we have the minimal cubature rule of degree $4n-1$ 
with $\dim \Pi_{2n-1}^2 + n$ nodes, 
\begin{align} \label{MinimalCuba-}
       \int_R f(u, v) & U_{-\frac12}(u,v) du dv = \frac12 \sum_{k=1}^n \mathop{ {\sum}' }_{j=1}^k   \l_k \l_j \\
   \times  & \sum f \left(\pm \cos \tfrac{\t_{j,n}}{2} \cos \tfrac{\t_{k,n}}{2}, 
      \pm \sin \tfrac{\t_{j,n}}{2}\sin \tfrac{\t_{k,n}}{2}\right),   \quad f \in \Pi_{4n-1}^2, \notag
\end{align}
where the innermost $\sum$ is a summation of four terms over all possible choices of signs. 
For $U_{\frac12}$ on $R$, we have the minimal cubature rule of degree $4n-3$ with 
$\dim \Pi_{2n-3}^2 + n$ nodes, 
\begin{align} \label{MinimalCuba+}
    \int_R f(u, v)& U_{\frac12}(u,v) du dv   = \frac12 \sum_{k=2}^n  \sum_{j=1}^{k-1} \l_{j,k} \\
  \times & \sum f \left(\pm \cos \tfrac{\t_{j,n}}{2} \cos \tfrac{\t_{k,n}}{2}, 
      \pm \sin \tfrac{\t_{j,n}}{2}\sin \tfrac{\t_{k,n}}{2}\right),   \quad f \in \Pi_{4n-3}^2. \notag
\end{align} 
\end{thm}

\begin{proof}
Changing variables $s = u^2$ and $t = v^2$ in \eqref{GaussCuba2-} and applying \eqref{IntR},
we obtain
\begin{align*}
  \frac{1}{4} \int_R f(u^2, v^2)  U_{-\frac12} (u,v) dudv 
       & =  \sum_{k=1}^n \mathop{ {\sum}' }_{j=1}^k  \l_k \l_j f( x_{j,k}^*, y_{j,k}^*) \\
       & =  \sum_{k=1}^n \mathop{ {\sum}' }_{j=1}^k \l_k \l_j f\left( \cos^2 \tfrac{\t_{j,n}}{2}\cos^2 \tfrac{\t_{k,n}}{2},
                   \sin^2 \tfrac{\t_{j,n}}{2}\sin^2 \tfrac{\t_{k,n}}{2}\right) 
\end{align*}
for all $f \in \Pi_{2n-1}$, where we have used the fact that $x^*_{j,k} =  \cos^2 \tfrac{\t_{j,n}}{2}
\cos^2 \tfrac{\t_{k,n}}{2}$, and $y^*_{j,k} = \sin^2 \tfrac{\t_{j,n}}{2}\sin^2 \tfrac{\t_{k,n}}{2}$, which 
follows from the definition of $x^*_{j,k}$ and $y^*_{j,k}$. The above cubature rule can be viewed 
as \eqref{MinimalCuba-} applied to $f(x^2,y^2)$. Since $\{f(x^2,y^2): f \in \Pi_{2n-1}^2\}$ consists
of all polynomials in $\Pi_{4n -1}^2$ that are invariant under the group $\ZZ_2 \times \ZZ_2$,
it implies, by the Sobolev theorem,  cubature rule \eqref{MinimalCuba-}. Since none of the nodes 
of \eqref{GaussCuba2-} are on the edges $s = 0$ or $t =0$ of $\Omega^*$, the number of nodes
of cubature rule \eqref{MinimalCuba-} is exactly
$$
   4 \dim \Pi_{n-1}^2 = 2n (n+1) = \dim \Pi_{2 n-1}^2 + n = \dim \Pi_{2n-1}^2 + \tfrac{2n}{2},
$$
which attains the lower bound in \eqref{lwbd}. The proof of the cubature rule \eqref{MinimalCuba+}
is similar. 
\end{proof}

The nodes of the cubature rules of degree 19 for $U_{-\frac12,-\frac12,-\frac12}$ are depicted in 
the right figure of Figure 3. 

As a final change of variables, we rotate the rhombus by $45^\circ$ to the square $[-1,1]^2$. This
amounts to a change of variables $u = (x+y)/2$ and $v= (x-y)/2$. The measure under this change of 
variables become $U_\g(u,v) dudv = \CW_\g (x,y) dxdy$, 
\begin{align*}% \label{V_g*}
    \CW_\g(x,y)  = &   b_w^\g 4^{\g} w(\cos (\t-\phi)) w(\cos  (\t+\phi)) |x^2 - y^2| (1-x^2)^\g (1-y^2)^\g,   \\ 
         & \qquad \hbox{where} \, \, x = \cos \t, \,\, y = \cos \phi, \quad
           (x,y) \in [-1,1]^2. \notag
\end{align*}
A simple computation shows that this is precisely the weight function defined in \eqref{CWgamma}. 
In the case of $U_{\a,\b,\g}$, the corresponding weight becomes $\CW_{\a,\b, \g}$ defined by 
\begin{align*} %\label{CWabc2nd}
   \CW_{\a,\b, \g}(x,y) = b_{\a,\b,\g} 4^{\g} |x+y|^{2\a +1}  |x- y|^{2\b +1}   (1-x^2)^\g (1-y^2)^\g, 
\end{align*}
which is exactly \eqref{CWabc}. Since the strength of the cubature rules do not change under the 
affine change of variables, we then have minimal cubature formulas for $\CW_{\a,\b,\g}$. To state 
this cubature explicitly, let us define 
\begin{align} \label{stjk}
  s_{j,k}: =  \cos \tfrac{\t_{j,n}-\t_{k,n}}{2} \quad\hbox{and}\quad
                t_{j,k} := \cos \tfrac{\t_{j,n} + \t_{k,n}}{2},
\end{align}
where $\t_{k,n}$ is again the angular argument of the zeros $x_{k,n} = \cos \t_{k,n}$ of $p_n$.

\begin{thm} \label{thm:cubaCW}
For $\CW_{-\frac12}$ on $[-1,1]^2$, we have the minimal cubature rule of degree $4n-1$ 
with $\dim \Pi_{2n-1}^2 + n$ nodes, 
\begin{align} \label{MinimalCuba2-}
  \int_{[-1,1]^2}  f(x, y) \CW_{-\frac12}(x,y) dx dy = & \frac12 \sum_{k=1}^n \mathop{ {\sum}' }_{j=1}^k 
     \l_k \l_j  \left[ f( s_{j,k}, t_{j,k})+  f( t_{j,k}, s_{j,k})    \right .  \\  
        & \quad + \left.  f( - s_{j,k}, - t_{j,k})+  f(- t_{j,k},- s_{j,k})  \right].      \notag
\end{align}
For $\CW_{\frac12}$ on $[-1,1]^2$, we have the minimal cubature rule of degree $4n-3$ with 
$\dim \Pi_{2n-3}^2 + n$ nodes, 
\begin{align} \label{MinimalCuba2+}
  \int_{[-1,1]^2}  f(x, y) \CW_{\frac12}(x,y) dx dy = & \frac12 \sum_{k=2}^n \sum_{j=1}^{k-1} 
     \l_{j,k}  \left[ f( s_{j,k}, t_{j,k})+  f( t_{j,k}, s_{j,k})    \right .  \\  
        & \quad + \left.  f( - s_{j,k}, - t_{j,k})+  f(- t_{j,k},- s_{j,k})  \right], \notag
\end{align} 
where $\lambda_{j,k} = \l_j \l_k (\cos \t_{j,n} - \cos \t_{k,n})^2$. 
\end{thm}

In the case of the product Jacobi weight function $\CW_{-\frac12, - \frac12,  \pm \frac12}$, these cubature 
rules were constructed in \cite{MP} and, more recently, in \cite{LSX} via a completely different method. 
In all other cases these cubature rules are new. The nodes of the cubature rule of degree 35 for the 
weight function 
\begin{align*}
     \CW_{- \frac12,- \frac12,- \frac12}(x,y)& = (1-x^2)^{-\frac12} (1-y^2)^{-\frac12}, \\
     \CW_{0,0,- \frac12}(x,y)& = |x^2 - y^2| (1-x^2)^{-\frac12} (1-y^2)^{-\frac12}, 
\end{align*}
are depicted in the left and right figures in Figure 4, respectively. 
\begin{figure} [ht] 
\includegraphics[scale=0.5]{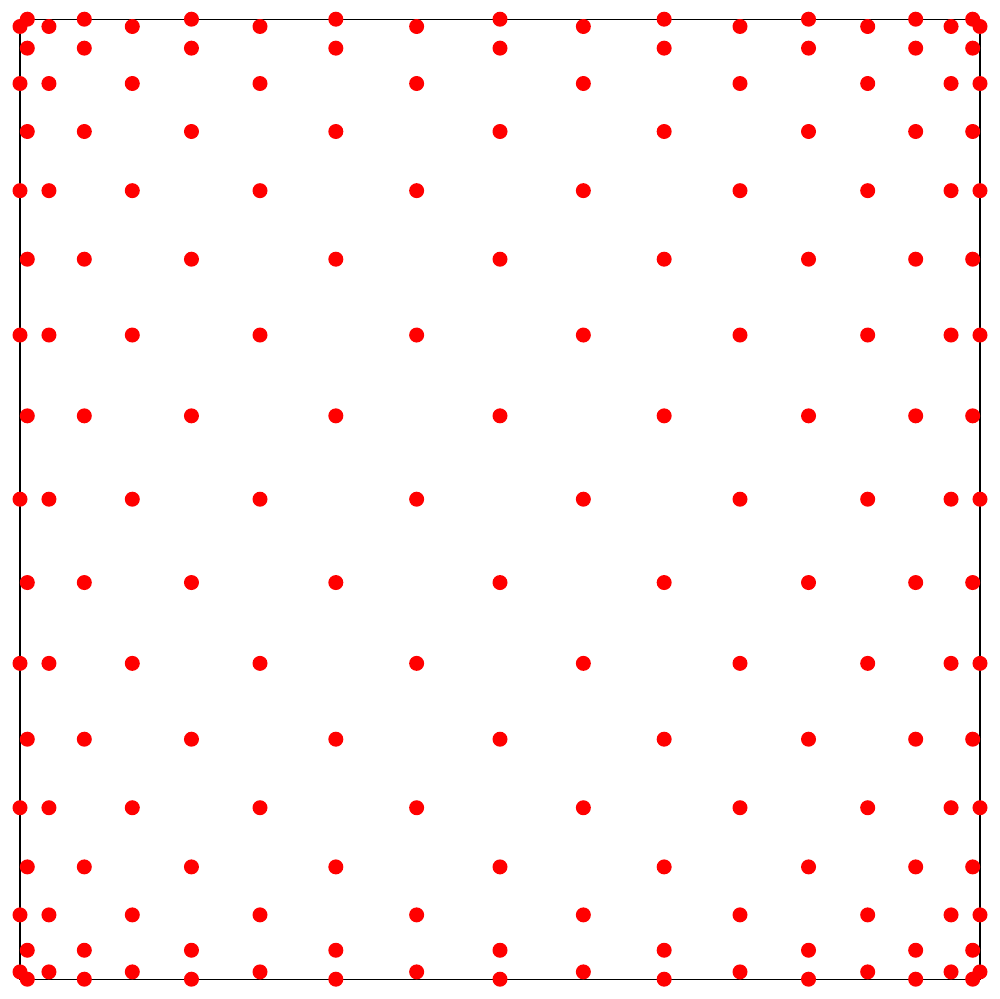} \qquad \includegraphics[scale=0.6]{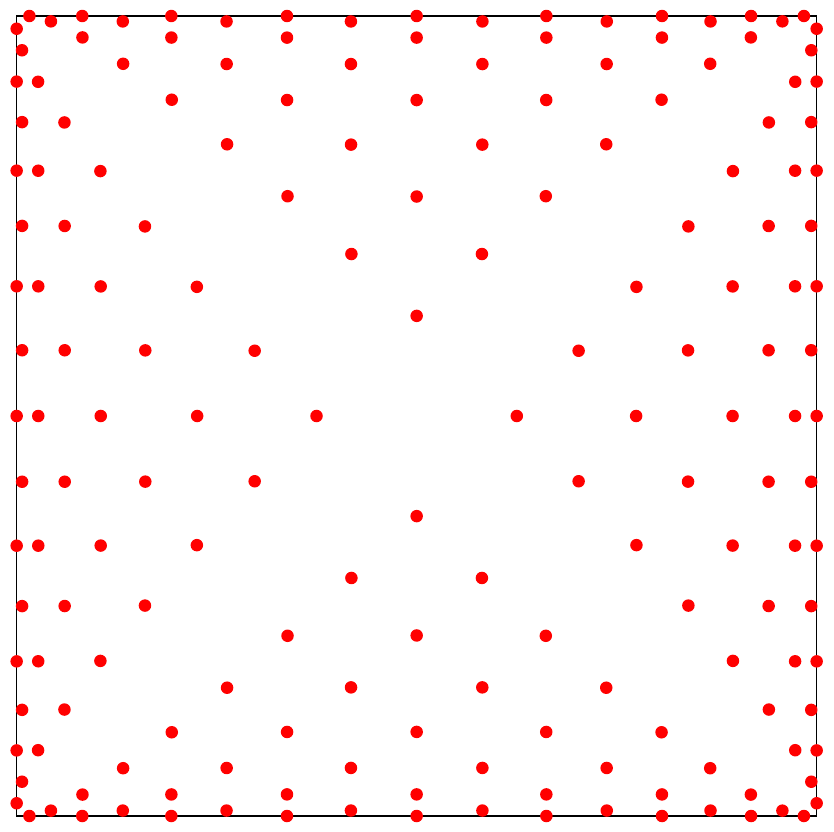}
\caption{Nodes of minimal cubature rules of degree 35 for $\CW_{-\frac12, -\frac12, - \frac12}$ and
   $\CW_{0,0,- \frac12}$ }
\label{figure:region4} 
\end{figure} 
The influence of the part $|x^2-y^2|$ in the weight function $\CW_{0,0,- \frac12}$ is clearly visible 
in comparing with the cubature rules for  $\CW_{-\frac12, -\frac12, - \frac12}$.  

By the relation \eqref{CWgamma} and the integral relation \eqref{Int-P-Q}, we could arrive at cubature 
rules \eqref{MinimalCuba2-} and \eqref{MinimalCuba2+} from those in \eqref{GaussCuba-} and 
\eqref{GaussCuba+} by the mapping $(x,y) \mapsto (2 xy, x^2+y^2 -1)$, bypassing some of the middle 
steps. Our presentation, on the other hand, is more intuitive and provides, hopefully, a better explanation 
of the connection between the Gaussian cubature rules for $W_\g$ and the minimal cubature rules 
for $\CW_\g$. 

We can also give a proof of Theorem \ref{thm:cubaCW} based on Theorem \ref{thm:minimalCuba}
by considering the common zeros of corresponding orthogonal polynomials, although a direct computation 
of the cubature weights will not be easy. Recalling the orthogonal polynomials ${}_1Q_{k,2n}^{(\g)}$ 
defined in \eqref{Qeven}, the following corollary is an immediate consequence of 
Theorem \ref{thm:minimalCuba}. 

\begin{cor}
The nodes of the minimal cubature rule \eqref{MinimalCuba2-} are the common zeros of 
orthogonal polynomials $\{{}_1Q_{k,2n}^{\a,\b}: 0 \le k \le n\}$ in Proposition \ref{Q2nJacobi}. And 
the nodes of the minimal cubature rule \eqref{MinimalCuba2+} are the common zeros of 
orthogonal polynomials $\{{}_1Q_{k,2n-2}^{\a,\b, \frac12}: 0 \le k \le n -1\}$. 
\end{cor} 

The relation \eqref{Qeven} shows that the nodes of the minimal cubature rule \eqref{MinimalCuba2-}
and the nodes of the Gaussian cubature rule \eqref{GaussCuba-} are related by a simple formula: 
if $(s,t)$ is a node of the former, then $(2st, s^2+t^2-1)$ is a node of the latter; furthermore, 
the nodes $(s_{j,k},t_{j,k}), (t_{j,k},s_{j,k}), (-s_{j,k},- t_{j,k}), (- t_{j,k},  -s_{j,k})$ of the former
correspond to the same node $(2 s_{j,k}t_{j,k}, s_{j,k}^2 + t_{j,k}^2-1)$ of the latter. This can also
be verified directly by elementary trigonometric identities. 

It should be pointed out that the Theorem \ref{NoGaussian} shows that the above construction 
does not work for the weight functions 
$$
       \CW_{-\frac12, - \frac12, \g} (x,y) = (1-x^2)^\g (1-y^2)^\g
$$
when $\g \ne \pm 1/2$. We cannot, however, conclude that the cubature rules of degree $4n-1$
that attain the lower bound \eqref{lwbd} do not exist for these product Gegenbauer weight  
functions. In fact, examining the proof carefully shows that the procedures that we adopted 
could be reversed only if the cubature rules for $\CW_{-\frac12, - \frac12, \g}$ satisfy certain 
properties. What we can conclude is then the following: If a cubature rule of degree $4n-1$ that attains the
lower bound \eqref{lwbd} exists for $\CW_{-\frac12, - \frac12, \g}$, then either it is not invariant under 
the symmetry with respect to the diagonals $y = x$ and $y = -x$ of the rectangle $[-1,1]^2$ or some
of its nodes are on these diagonals. 

Finally, our procedure of deriving cubature rules for $\CW_{\pm \frac12}$ on $[-1,1]^2$ can be 
adopted for other type of cubature rules, such as cubature rules of even degree or Gauss-Lobatto 
type cubature rules, see the remark at the end of Subsection 3.1. In particular, if we start with 
a Gaussian-Lobatto quadrature for $w$, which has additional nodes at $-1$ and $1$, then the 
resulted cubature rule for $\CW_\g$ will have nodes on the diagonals of $[-1,1]^2$. Since they 
do not seem to have other features, we shall not pursue them further. 

\section{Lagrange interpolation and Gaussian cubature rules}
\setcounter{equation}{0}

Cubature rules are closely related to Lagrange interpolation polynomials, as stated in
Section 2. In this section we consider Lagrange interpolation polynomials based on the zeros 
of the Gaussian cubature rules constructed in the Section 3. 
 
The Lagrange interpolation polynomial based on the Gaussian cubature rule in 
Theorem \ref{thm:Gauss} is given in Theorem \ref{thm:Gauss-Interp}. A more 
direct construction can be given however as follows. 

Let $\{x_{k,n}: 1 \le k \le n \}$ be the zeros of the orthogonal polynomial $p_n$ of degree $n$
with respect to $w$ on $[-1,1]$, as in \eqref{Gauss[-1,1]}. The Lagrange interpolation 
polynomial (of one variable) of degree $< n$ based on these points is
\begin{equation}\label{interp-1d}
    I_n f(x) = \sum_{k=1}^n f(x_{k,n}) l_k(x), \qquad l_k(x) := \frac{p_n(x)}{p_n'(x_{k,n}) (x-x_{k,n})}.    
\end{equation}
Recall that $u_{j,k} = x_{j,n}+ x_{k,n}$ and $v_{j,k}  =x_{j,n}x_{k,n}$. 

\begin{thm}
The unique Lagrange interpolation polynomial of degree $n-1$ based on the nodes of the Gaussian 
cubature rule \eqref{GaussCuba-} is given by 
\begin{align}\label{interp-Gauss-}
    & L_n f(u,v) = \sum_{k=1}^n \mathop{ {\sum}' }_{j=1}^k   f( u_{j,k}, v_{j,k}) l_{j,k}(u,v), \\
    &    \hbox{with}\,\,  l_{j,k}(u,v) := l_j(x) l_k(y) + l_j(y) l_k(x), \quad u = x+y, \,\, v = xy.  \notag
\end{align}
And the unique Lagrange interpolation polynomial of degree $n-2$
based on the nodes of the Gaussian cubature rule \eqref{GaussCuba+} is given by 
\begin{align}\label{interp-Gauss+}
   & L_n f(u,v) = \sum_{k=2}^n \sum_{j=1}^{k-1} f( u_{j,k}, v_{j,k}) l_{j,k}(u,v), \\
       &  \hbox{with}\,\,    l_{j,k}(u,v) := \frac{l_j(x) l_k(y) - l_j(y) l_k(x)}{x-y}, 
        \quad u = x+y, \,\, v = xy.  \notag 
\end{align}
\end{thm}

\begin{proof}
A quick computation shows that if $0 \le  j < k \le n$ and $0\le p \le q \le n$, then 
$$
   l_{j,k}(u_{p,q},v_{p,q}) = l_j(x_{p,n}) l_k(x_{q,n}) + l_j(x_{q,n}) l_k(x_{p,n})
        = \delta_{j,p} \d_{k,q} + \d_{k,p} \d_{j,q} = \d_{j,p} \d_{k,q}. 
$$
If $0 \le j = k \le n$ and $0\le p \le q \le n$, then 
$$
   l_{j,k}(u_{p,q},v_{p,q}) =  2 l_k(x_{p,n}) l_k(x_{q,n}) = 2 \d_{j,p} \d_{k,q}, 
$$
which proves \eqref{interp-Gauss-}. The proof of \eqref{interp-Gauss+} is similar. 
\end{proof}

The explicit formulas of $l_{j,k}$ can also be obtained from Theorem \ref{thm:Gauss-Interp}.
In fact, as shown in \cite[Theorem 3.1]{X10}, the reproducing kernel 
$K_n^{(\pm \frac12)}(\cdot,\cdot) = K_n(W_{\pm \frac12}; \cdot,\cdot)$ can be expressed
in terms of the reproducing kernel 
$$
       k_n(x,y)= k_n(w;x,y) := \sum_{k=0}^n p_k(x) p_k(y)
$$
of one variable, where $p_k$ are orthonormal polynomials with respect to $w$. Set 
$$
  u := (u_1,u_2) = (x_1+x_2, x_1 x_2) \quad  \hbox{and}\quad  v:= (v_1,v_2) = ( y_1+y_2, y_1 y_2).
$$
The reproducing kernel $K_n^{(-\frac12)}(\cdot, \cdot)$ for $W_{- \frac12}$ is given by 
\begin{equation} \label{eq:reprodW-1/2}
    K_n^{(-\frac12)}(u,v) =  \frac12 \left[ k_n(x_1,y_1)k_n(x_2,y_2)+ k_n(x_2,y_1) k_n(x_1,y_2) \right], 
\end{equation}
and the reproducing kernel $K_n^{(\frac12)}(\cdot, \cdot)$ for $W_{\frac12}$ is given by 
\begin{align} \label{eq:reprodW+1/2}
 K_n^{(\frac12)}(u,v) = \frac{k_{n+1}(x_1,y_1)k_{n+1}(x_2,y_2) - 
   k_{n+1}(x_2,y_1) k_{n+1}(x_1,y_2)} {2 (x_1-x_2)(y_1-y_2)}. 
\end{align}

As an application of the explicit expression, we can estimate the uniform norm of 
the interpolation operator, often called the Lebesgue constant. For the interpolation
polynomial $I_n f$ in \eqref{interp-1d}, the Lebesgue constant $ \|I_n\|_{C[-1,1]}$ satisfies 
$$
      \|I_n\|_{C[-1,1]} =  \max_{x \in [-1,1]} \sum_{k=1}^n |l_k(x)|. 
$$ 

\begin{cor} \label{LebesgueL}
The Lebesgue constant for $L_n f$  in \eqref{interp-Gauss-} satisfies
\begin{equation} \label{Lebesgue}
     \|L_n \|_\infty  \le \left(\|I_n\|_{C[-1,1]} \right)^2. 
\end{equation} 
\end{cor}

\begin{proof}
A standard argument shows that the Lebesgue constant for $\CL_n f$ is given by 
$$
 \|L_n \|_\infty = \max_{ (u,v) \in \Omega}  \sum_{k=1}^n \mathop{ {\sum}' }_{j=1}^k |l_{j,k}(u,v)|. 
$$
Since $\ell_{j,k}(u,v) = \ell_{k,j}(u,v)$ by \eqref{interp-Gauss-}, a moment of reflection shows that 
$$
  \sum_{k=1}^n \mathop{ {\sum}' }_{j=1}^k | l_{j,k}(u,v)| = \sum_{k=1}^n \sum_{j=1}^n 
    | l_j(x) l_k(y) +  l_j(x) l_k(y)| \le 2 \sum_{j=1}^n| l_j(x)| \sum_{j=1}^n | l_j(y)|, 
$$
from which the estimate \eqref{Lebesgue} follows immediately.  
\end{proof}

Denote by  $L_n^{\a,\b} f$ the Lagrange interpolation polynomial based 
on the nodes of the Gaussian cubature rule of degree $2n-1$ for $W_{\a,\b,-\frac12}$. 

\begin{cor} \label{LebesgueL-Jacobi}
Let $\a, \b > -1$. The Lebesgue constant of $L_n^{\a,\b} f$ satisfies 
$$
  \|L_n^{\a,\b}\|_\infty = \CO(1) \begin{cases} n^{2 \max \{\a,\b\} + 1}, & \max\{\a,\b\} > -1/2, \\
                               \log^2 n, & \max\{\a,\b\} \le -1/2.  \end{cases}
$$
\end{cor}

\begin{proof}
This follows from the previous corollary and the classical result on the Lagrange interpolation 
polynomials at the zeros of Jacobi polynomials \cite{Sz}. 
\end{proof}

\section{Lagrange interpolation and minimal cubature rules}
\setcounter{equation}{0}

The relation between a minimal cubature rule and the Lagrange interpolation polynomial
based on its nodes is stated in Theorem \ref{thm:minimal-Interp}. In this section we discuss
the Lagrange interpolation polynomials based on the nodes of the minimal cubature rules of 
degree $4n-1$ in Section 4. In order to derive explicit formulas and discuss the Lebesgue
constants, we shall limit our discussion to $\CW_{\a,\b, - \frac12}$, which we renamed as 
$\CW_{\a,\b}$ at \eqref{CWabab}. An analogue discussion can be carried out for  $\CW_{\a,\b, \frac12}$.

\subsection{Construction of the interpolation polynomial}
Let $X_n$ denote the set of nodes of the cubature formula $\CW_{\a,\b}$. The Lagrange 
interpolation  polynomial based on $X_n$ is given in Theorem \ref{thm:minimal-Interp}, in
which $x_{k,n} = x_{k,n}^{(\a,\b)}$ are the zeros of Jacobi polynomial $P_n^{(\a,\b)}$. 
The subspace $\Pi_{2n}^*$ in \eqref{Pin*} now takes the form
$$ 
 \Pi_{2n}^* : = \Pi_{2n-1}^2 \cup \mathrm{span} \{{}_2Q_{k,2n}^{( \pm \frac12)}: 0 \le k \le n-1\}.
$$ 
The interpolation polynomial in $\Pi_{2n}^*$ is given in Theorem \ref{thm:minimal-Interp} in 
terms of a kernel of $\Pi_{2n}^*$ defined by 
\begin{equation} \label{wtCKn}
  \CK_{2n}^* (x,y) : = \CK_{2n-1}^{\a,\b}(x,y) + 
           \sum_{k=0}^{n-1} b_{k,n} {}_2Q_{k,2n}^{\a,\b}(x) {}_2Q_{k,2n}^{\a,\b}(y),
\end{equation}
where $b_{k,n}$ are certain positive numbers, $\CK_{2n-1}^{\a,\b}$ and ${}_2Q_{k,2n}$ are 
given explicitly in \eqref{eq:reprodCW} and Proposition \eqref{Q2nJacobi}. 

Although the cubature rule \eqref{MinimalCuba2-} of degree $4n-1$ for $\CW_{-\frac12}$ can be 
deduced from the Gaussian cubature rule \eqref{GaussCuba-} for $W_{-\frac12}$, this deduction 
does not extend to interpolation polynomials, since each node of the cubature rule \eqref{GaussCuba-} 
corresponds to four nodes of the cubature rule \eqref{MinimalCuba2-}. We have to work with the 
explicit formula given in Theorem \ref{thm:minimal-Interp}, which we determine explicitly in the
following theorem. 

\begin{thm} \label{interp-CL*}
Let $x_{k,n} = \cos \t_{k,n}$, $1 \le k \le n$, denote the zeros of the Jacobi polynomial $P_n^{(\a,\b)}$
and let $s_{j,k}: =  \cos \tfrac{\t_{j,n}-\t_{k,n}}{2}$ and  $t_{j,k} := \cos \tfrac{\t_{j,n} + \t_{k,n}}{2}$.
Set 
\begin{align*}
    \xb_{j,k}^{(1)}:=(s_{j,k},t_{j,k}), \,\,  \xb_{j,k}^{(2)} :=(t_{j,k},s_{j,k}), \,\, 
      \xb_{j,k}^{(3)}:=(-s_{j,k},- t_{j,k}), \,\, \xb_{j,k}^{(4)}:= (-t_{j,k},-s_{j,k}). 
\end{align*}
Then the Lagrange interpolation $\CL_n^{\a,\b} f$ in $\Pi_n^*$ is given by 
\begin{align} \label{CLnWab}
   \CL_n^{\a,\b} f(x,y) = \sum_{k=1}^n   \sum_{j=1}^k & \left[ f\left(\xb_{j,k}^{(1)}\right) \ell^{(1)}_{j,k} (x,y)
         + f\left(\xb_{j,k}^{(2)}\right) \ell^{(2)}_{j,k} (x,y) \right. \\
         & \left. + f\left(\xb_{j,k}^{(3)}\right) \ell^{(3)}_{j,k} (x,y)+f\left(\xb_{j,k}^{(4)}\right) \ell^{(4)}_{j,k} (x,y) \right], \, \notag
\end{align}
where the fundamental interpolation polynomials $\ell^{(i)}_{j,k}$ are given by
\begin{align} \label{ell-ijk}
  \ell^{(i)}_{j,k} (x,y) = \frac12 \l_j^{(\a,\b)}\l_k^{(\a,\b)} \CK_{2n}^*\left((x,y), \xb_{j,k}^{(i)}\right),
\end{align}
in which $\frac12$ in the right hand side needs to be replaced by $\frac14$ when $j = k$, and  
\begin{align} \label{eq:CK*Jacobi}
    \CK_{2n}^* (x,y)  =    \CK_{2n-1}^{\a,\b}(s,t) \,  + &   \frac{1 + \a + \b + n}{1 + \a + \b + 2 n} 
        d_{\a,\b}^{(1,1)} (x_1^2-x_2^2)(y_1^2-y_2^2) \\
            & \times   \left[ K_{n-1}^{\a+1,\b+1} (s,t) - K_{n-2}^{\a+1,\b+1} (s,t) \right] \notag \\
        - & \frac{n (1 + \a + \b + n)}{ (1 + \a + \b + 2 n)^2} {}_2Q_{n-1,2n}(x){}_2Q_{n-1,2n}(y), \notag
\end{align}
where $s= (2 x_1 x_2,x_1^2+x_2^2-1)$, $t = (2 y_1 y_2, y_1^2+y_2^2 -1)$, $\CK_{2n-1}^{\a,\b}(\cdot,\cdot)$ 
and $d_{\a,\b}^{(1,1)}$  are given in \eqref{eq:reprodCW} and in $K_n^{\a,\b}(\cdot,\cdot)$ is given in 
\eqref{Kn-st-1/2}.
\end{thm}

\begin{proof}
The formulas \eqref{CLnWab} and \eqref{ell-ijk} are exactly those given in Theorem \ref{thm:minimal-Interp},
specialized to the Jacobi case. It remains to establish the formula of \eqref{eq:CK*Jacobi}, for 
which we need to determine the constants $b_{k,n}$ in \eqref{wtCKn}. 

Throughout this proof, we write $Q_{k,2n}(x,y) = {}_2Q_{k,n}^{(\a,\b)}(x,y)$. By the explicit formula 
of $Q_{k,2n}$ in Proposition \ref{Q2nJacobi}, it is easy to verify that 
\begin{align} \label{Qxjk}
  &  Q_{m,2n}\left(\xb_{j,k}^{(1)}\right) = \gamma_{\a,\b} \sqrt{1-x_j^2}\sqrt{1-x_k^2}  \\
     & \quad  \times \left[ p_{n-1}^{(\a+1,\b+1)}(x_k) p_m^{(\a+1,\b+1)}(x_j)+
                  p_{n-1}^{(\a+1,\b+1)}(x_j) p_m^{(\a+1,\b+1)}(x_k)\right], \notag
\end{align}
and furthermore, since $Q_{m,2n}$ is symmetric in its variables, 
\begin{equation} \label{Qxjk2}
    Q_{m,2n}\left(\xb_{j,k}^{(1)}\right) =  Q_{m,2n}\left(\xb_{j,k}^{(1)}\right) =  -Q_{m,2n}\left(\xb_{j,k}^{(3)}\right) 
     = -Q_{m,2n}\left(\xb_{j,k}^{(3)}\right). 
\end{equation}
Let us denote by $\C_n[f]$ the minimal cubature rule, that is,  
$$
   \C_n [f]:= \frac12 \sum_{k=1}^n \mathop{ {\sum}' }_{j=1}^k  
     \l_k^{(\a,\b)} \l_j^{(\a,\b)}  \left[  f\left(\xb_{j,k}^{(1)}\right)+   f\left(\xb_{j,k}^{(2)}\right) 
          +  f\left(\xb_{j,k}^{(3)}\right)+   f\left(\xb_{j,k}^{(4)} \right) \right].
$$           
By \eqref{mcfWeight} and the fact that $\ell_{j,k}^{(i)}$ are the fundamental interpolation polynomials, we obtain
\begin{equation*}
     \CK_{2n}^*\left(\xb_{j,k}^{(1)}, \xb_{j',k'}^{(1)}\right) = 2 \left(\lambda_{j}^{(\a,\b)}\lambda_k^{(\a,\b)}\right)^{-1} 
            \delta_{j,j'}\delta_{k,k'}, 
\end{equation*}
which implies immediately that 
\begin{equation} \label{eqn-bmn}
     \C_n \left[\CK_{2n}^*\left(\xb_{j,k}^{(1)}, \cdot\right) Q_{l, 2n}\right] =  Q_{l, 2n}\left(\xb_{j,k}^{(1)}\right).
\end{equation}
On the other hand, using the formula of $\CK_{2n}^*(\cdot,\cdot)$ in \eqref{wtCKn} shows that  
$$
  \C_n \left[\CK_{2n}^*\left(\xb_{j,k}^{(1)}, \cdot\right) Q_{l, 2n}\right] 
     =  \C_n \left[K_{2n-1}^{\a,\b} \left(\xb_{j,k}^{(1)}, \cdot\right) Q_{l, 2n}\right]
         + \sum_{m=0}^{n-1} b_{m,n}  \C_n \left[ Q_{m,2n} Q_{l, 2n}\right].
$$
Since the cubature rule is of degree $4n-1$ and $Q_{l,2n}$ is an orthogonal polynomial of degree
$2n$,  
$$
 \C_n \left[K_{2n-1}^{\a.\b} \left(\xb_{j,k}^{(1)}, \cdot\right) Q_{l, 2n}\right] = 
   2 c_{\a,\b}^2 \int_{[-1,1]^2} K_{2n-1}^{\a,\b} \left(\xb_{j,k}^{(1)}, y \right) Q_{l, 2n}(y)W_{\a,\b}(y) dy= 0.
$$
Furthermore, since $Q_{m,2n}$ is symmetric in its variables, it follows from \eqref{Qxjk2} that
\begin{align*}
   \C_n \left[ Q_{m,2n} Q_{l, 2n}\right]  = \sum_{k=1}^n \sum_{j=1}^n \l_k^{(\a,\b)}\l_j^{(\a,\b)}
       Q_{m,2n}\left(\xb_{j,k}^{(1)}\right) Q_{l,2n}\left(\xb_{j,k}^{(1)}\right). 
\end{align*} 
Recall the definition of $\wh h_m$ defined in \eqref{hathn}. By \eqref{Qxjk2}, the explicit formulas
of $Q_{k,2n}$ and the 
Gaussian quadrature \eqref{Gauss-quadrature}, it follows that 
\begin{align*}
 \C_n \left[ Q_{m,2n} Q_{l, 2n}\right]  = 2 \g_{\a,\b}^2 \wh h_{n-1} \wh h_m \delta_{l,m}, 
    \quad 0 \le l, m \le n-1. 
\end{align*} 
Putting these formulas together, we have shown that 
$$
\C_n \left[\CK_{2n}^*\left(\xb_{j,k}^{(1)}, \cdot\right) Q_{l, 2n}\right] 
    =2 \g_{\a,\b}^2 \wh h_{n-1} \wh h_m  b_{l,n} Q_{l, 2n}\left(\xb_{j,k}^{(1)}\right). 
$$
Comparing with \eqref{eqn-bmn}, it follows readily that $b_{l,n}^{-1}  = 2 \g_{\a,\b}^2 
\wh h_{n-1} \wh h_m$. Recalling that $\g_{\a,\b} = c_{\a+1,\b+1}/(\sqrt{2}c_{\a,\b})$,
applying Lemma \ref{hat-hn} gives 
$$
   b_{0,n} = \cdots = b_{n-2,n} =   \frac{1 + \a + \b + n}{1 + \a + \b + 2 n} , \quad
   \hbox{and} \quad b_{n-1,n} = b_{0,n}^2.
$$
The final step in verifying \eqref{eq:CK*Jacobi} uses the fact that 
\begin{align*}
 & \sum_{k=0}^{n-1} Q_{k,2 n}(x) Q_{k,2n}(y) =  d_{\a,\b}^{(1,1)} (x_1^2-x_2^2)(y_1^2-y_2^2) 
              \left[ K_{n-1}^{\a+1,\b+1} (s,t) - K_{n-2}^{\a+1,\b+1} (s,t) \right]
\end{align*} 
which can be verified using the explicit formulas of the quantities involved and the elementary
trigonometric identity 
\begin{equation} \label{theta-phi}
    2 x y = \cos (\t - \phi) + \cos (\t + \phi), \quad x^2+y^2 -1 = \cos (\t -  \phi) \cos (\t + \phi), 
\end{equation} 
see also Section 4 of \cite{X10}. This completes the proof. 
\end{proof}

The above theorem gives a compact formula for the Lagrange interpolation polynomial based
on the nodes of the minimal cubature rule with respect to $\CW_{\a,\b}$. In the case of $\a = \b 
= -1/2$, the interpolation polynomials were introduced in \cite{X96} and they were studied 
numerically in \cite{BCMV}. The explicit formulas given in \cite{X96}, however, takes a different 
form since the set of nodes were not divided into the four subsets as in \eqref{CLnWab} and 
a completely different formula for $\CK_{2n}^*(\cdot,\cdot)$ was used.  

\subsection{Lebesgue constants of the interpolation operator}
The Lebesgue constant of the interpolation operator $\CL_n^{\a,\b}$ is its operator norm 
$ \|\CL_n^{(\a,\b)}\|_\infty$. Since 
$$
    \|\CL_n^{\a,\b} f \|_\infty \le  \|\CL_n^{(\a,\b)}\|_\infty \|f\|_\infty, \quad \forall f \in C[-1,1]^2, 
$$  
the Lebesgue constant determines the convergence behavior of $\CL_n^{\a,\b} f$. 

\begin{lem} 
The Lebesgue constant of $\CL_n^{\a,\b} f$ in Proposition \ref{interp-CL*} satisfies 
\begin{align} \label{LebesgueCW}
    \|\CL_n^{\a,\b}\|_\infty =  \frac{1}{4} \max_{x \in [-1,1]^2} \sum_{k=1}^n \mathop{ {\sum}' }_{j=1}^k
     \l_{j,k}  & \left[  \left|\CK_{2n}^*\left( x, \xb_{j,k}^{(1)}\right)\right|  +  \left|\CK_{2n}^*\left(x, \xb_{j,k}^{(2)}\right)\right|  \right . \\
                  & \left.  +   \left |\CK_{2n}^*\left(x, \xb_{j,k}^{(3)}\right)\right| +
            \left|\CK_{2n}^*\left(x, \xb_{j,k}^{(4)}\right)\right|   \right] \notag \\
             \sim  \max_{x \in [-1,1]^2} \sum_{k=1}^n \mathop{ {\sum}' }_{j=1}^k 
     \l_{j,k}  &  \left|\CK_{2n}^*\left( x, \xb_{j,k}^{(1)}\right)\right|.  \notag
\end{align}
\end{lem}

\begin{proof}
Recalling \eqref{stjk} and the definition of $\xb_{j,k}^{(i)}$, it follows easily from the symmetry that
$\|\CL_n^{\a,\b}\|_\infty$ is bounded above by 4 times of the quantity in the last expression and it is at 
least as big as the same quantity. 
\end{proof}

In order to deduce the order of the Lebesgue constant, we need to estimate, by the explicit formula at 
\eqref{eq:CK*Jacobi}, several sums. We first deal with the easiest sum to be estimated.
Let $c$ denote a generic constant whose value may vary from line to line. 

\begin{lem} \label{lem:LambdaQ}
For $\a,\b > -1$ and $x \in [-1,1]$, 
$$
 \Lambda_Q:= \sum_{k=1}^n \sum_{j=1}^k \l_k^{(\a,\b)} \l_j^{(\a,\b)} \left | {}_2Q_{n-1,2n}(x)  {}_2Q_{n-1,2n}\left( x, \xb_{j,k}^{(1)}\right)
       \right| \le c\, n^{2 \max \{\a,\b\}}. 
$$
\end{lem}

\begin{proof}
We will need several well known estimates for the Jacobi polynomials and related quantities, all can be
found in \cite{Sz}. First we need 
\begin{align}\label{JaobiBd}
      |p_n^{(\a,\b)}(x)| \le c \left(\sqrt{1-x} + n^{-1}\right)^{- (\a+1/2)/2}\left(\sqrt{1+x} + n^{-1}\right)^{- (\b+1/2)/2}
\end{align}
for $x \in [-1,1]$. Using the fact that $\cos^2 \t - \cos^2 \phi = \sin(\t - \phi) \sin(\t+\phi)$, it follows
from the explicit expression of ${}_2Q_{n-1,n}(x)$ that 
$$
      \left |{}_2Q_{n-1,2n}(x) \right| \le c \max_{-1 \le x \le 1} |\sqrt{1-x^2}   p_{n-1}^{(\a,\b)} (x)| \le c\, n^{2 \max \{\a,\b\} -1},
         \quad x \in [-1,1]^2.
$$
Furthermore, we need the estimates 
\begin{align}\label{lkn}
   \l_{k,n}^{(\a,\b)}   = \left[k_n^{(\a,\b)}(x_{k,n}, x_{k,n})\right]^{-1}& \sim n^{-1} w_{\a,\b}(x_{k,n}) \sqrt{1-x_{k,n}^2}, \\ 
      p_{n-1}^{(\a,\b)} (x_{k,n}) & \sim \left[w_{\a,\b}(x_{k,n})\right]^{-1} (1-x_{k,n})^{-1/4}.  \label{pn-1zero}
\end{align}
From \eqref{pn-1zero}, it is not difficult to see (using \eqref{Jacobi-property}, for example) that
$$
   \left |{}_2Q_{n-1,2n}\left( x, \xb_{j,k}^{(1)}\right) \right| \sim n  \left[w_{\a,\b}(x_{k,n})\right]^{-1/2} (1-x_{k,n})^{-1/4} 
    \left[w_{\a,\b}(x_{j,n})\right]^{-1/2} (1-x_{j,n})^{-1/4}.
$$
Consequently, since ${}_2Q_{n-1,2n}$ is symmetric in its variables, we see that 
\begin{align*}
  \Lambda_Q  \le c\, n^{2 \max \{\a,\b\} } \left( \sum_{k=1}^n \l_k^{(\a,\b)} \left[w_{\a,\b}(x_{k,n})\right]^{-1/2} (1-x_{k,n})^{-1/4} \right)^2
      \le  c\, n^{2 \max \{\a,\b\} }
\end{align*}
as the sum is easily seen to be bounded upon using \eqref{lkn}.
\end{proof}

The other sums of $\|\CL_n^{\a,\b}\|_\infty$ cannot be deduced form the Lebesgue constant for the 
interpolation polynomial of one variable, as we did in Corollary \ref{LebesgueL}, since there are
four remaining sums by \eqref{eq:CK*Jacobi}, and only one of them, the first one, is related directly
to the fundamental interpolation polynomials of one variable. We can, however, reduce the proof 
to the estimate of several kernels in one variable. Let us define, for $i,j \ge 0$,  
$$
    k_n^{(\a,\b),i,j}(x,y) := (1-x)^{\frac{i}{2}} (1+ x)^{\frac{j}{2}}(1-y)^{\frac{i}{2}} (1+ y)^{\frac{j}{2}}
         k_n^{(\a+i,\b+j)}(x,y). 
$$

\begin{lem}
Let $\a,\b \ge  -1/2$ and $i,j \ge 0$. Then
\begin{align} \label{kernel-esti}
  | k_n^{(\a,\b),i,j}(\cos \t, \cos \phi) |  \qquad\qquad\qquad\qquad \qquad\qquad\qquad\qquad \qquad\qquad\qquad  \\
  \le c   \frac{(\sin \tfrac{\t}{2}\sin \tfrac{\phi}{2} +  n^{-1} |\t-\phi| + n^{-2})^{-\a- \frac12}   
 (\cos \tfrac{\t}{2}\cos \tfrac{\phi}{2} + n^{-1} |\t-\phi| + n^{-2})^{-\b  - \frac12}}{|\t - \phi| + n^{-1}}. \notag  
\end{align}
\end{lem}

While \eqref{lkn} and \eqref{lkn} are classical, \eqref{kernel-esti} is stated recently in \cite[Lemma 5.3]{X10}
and its proof follows from an estimate in \cite{DaiX}. The restriction $\a,\b \ge -\frac12$ instead of 
$\a, \b > -1$ comes from the method used in \cite{DaiX}. For $i,j \ge 0$, let 
$$
     \Lambda_n^{(i,j)}(x): = \sum_{k=1}^n   \l_{k,n}^{(\a,\b)} \left | k_n^{(\a,\b),i,j}(x ,x_{k,n}) \right |.
$$
We will need the following result for our estimate of $\|\CL_n^{\a,\b}\|_\infty$. 

\begin{lem} \label{lem:Lnij}
Let $\a,\b \ge -1/2$. For $i,j \ge 0$, 
\begin{equation} \label{Lambda-ij}
    \max_{x \in [-1,1]} \Lambda_n^{(i,j)}(x) 
       = \CO(1) \begin{cases} n^{\max \{\a,\b\} + \frac12}, & \max\{\a,\b\} > -1/2, \\
                              \log n, & \max\{\a,\b\} = -1/2.  \end{cases}
\end{equation}
\end{lem}

\begin{proof}
We can assume $x \in [0,1]$ and write $x  = \cos \t$. We consider $\a > -1/2$, the case $\a = -1/2$ 
is easier. Fix $m$ such that $x_{m,n}$ is (one of) the closest zero to $x$. Then $1 \le m \le n/2 +1$. 
We only consider the sum in $\Lambda_n^{(i,j)}$ for $1 \le k \le 2n/3$, the remaining part is easier
since for $2n/3 < k \le n$, $|\t - \t_{k,n}| \sim 1$. If $k = m-1, m, m+1$, then by \eqref{kernel-esti} 
and \eqref{lkn}, 
$$
   \l_{k,n}^{(\a,\b)} \left | k_n^{(\a,\b)}(x ,x_{k,n}) \right | \le \frac {(\sin \t_{k,n})^{\a + \frac12}}
      {(\sin^2 \frac{\t_{k,n}}{2} + n^{-2})^{\a + \frac12}} \le c n^{\a+\frac12}. 
$$
Using the fact that $|\t - \t_k|  \sim |\t_m - \t|$, we have by \eqref{kernel-esti} and \eqref{lkn}
$$
     \sum_{\substack{|k - m|> 1 \\ 1\le k \le 2n/3}}  \l_{k,n}^{(\a,\b)} \left | k_n^{(\a,\b)}(x ,x_{k,n}) \right |
       \le c n^{\a+\frac12} \sum_{\substack{|k - m|> 1 \\ 1\le k \le 2n/3}}
            \frac{k^{\a + \frac12}}{|k-m|(k m + |k-m|)^{-\a-\frac12}}.  
$$
The last sum can be shown to be bounded by dividing it into three sums over 
$1 \le k \le m/2$, $m/2 \le k \le 2m$ and $m \le k \le 2n/3$, respectively.  Such estimates
are rather standard affairs, we leave the details to the interested readers. 
\end{proof}

For $i = j = 0$, the estimate \eqref{Lambda-ij} gives the order of the Lebesgue constant  
for the interpolation polynomials based on the zeros of Jacobi polynomials in one variable. 
The classical proof in \cite{Sz}, however, does not apply to the case of $(i,j) \ne (0,0)$, since
$\l_{k,n}^{(\a,\b)}  k_n^{(\a,\b),i,j}(x ,x_{k,n})$ does not always vanish at $x_{l,n}$ when $l \ne k$. 

We are now ready to prove our result on the Lebesgue constant of $\CL_n^{\a,\b} f$. 

\begin{thm}
Let $\a, \b \ge -1/2$. The Lebesgue constant of the Lagrange interpolation polynomial 
$\CL_n^{\a,\b}f$ based on the nodes of the minimal cubature rule of degree $4n-1$ for 
$\CW_{\a,\b}$ satisfies 
\begin{equation} \label{LebesgueCLn}
  \|\CL_n^{\a,\b}\|_\infty = \CO(1) \begin{cases} n^{2 \max \{\a,\b\} + 1}, & \max\{\a,\b\} > -1/2, \\
                                (\log n)^2, & \max\{\a,\b\} = -1/2.  \end{cases}
\end{equation}
\end{thm}

\begin{proof}
Let $x_{k,n} = \cos \t_{k,n}$ be the zeros of the Jacobi polynomial $p_n^{(\a,\b)}$. We estimate 
$\|\CL_n^{\a,\b}\|_\infty$ in \eqref{LebesgueCW} by setting $x_1 = \cos \frac{\t_1 - \t_2}{2}$ and 
$x_2 = \cos \frac{\t_1 + \t_2}{2}$ and taking the maximum over $0 \le \t_1, \t_2 \le \pi$. It follows
that $2x_1x_2 = \cos \t_1 + \cos \t_2$ and $x_1^2+x_2^2 -1 = \cos \t_1\cos \t_2$, and furthermore,
$$
  x_1-x_2 = \sqrt{(1-\cos \t_1)(1-\cos \t_2)} \quad \hbox{and}\quad   x_1+x_2 = \sqrt{(1+\cos \t_1)(1+\cos \t_2)}.
$$
Hence, recalling \eqref{stjk}, it follows from \eqref{eq:CK*Jacobi}, \eqref{Kn-st-1/2} and Lemma 
\ref{lem:LambdaQ} that 
\begin{align*}
  \|\CL_n^{\a,\b}\|_\infty = \CO(1) \max_{0 \le \t_1,\t_2 \le \pi} & \sum_{k=0}^n \mathop{ {\sum}' }_{j=1}^k \l_j \l_k 
       \left[  \left | J_{i,k}^{0,0}(\t_1, \t_2) \right| + \left | J_{j,k} n^{1,0}(\t_1, \t_2) \right| \right. \\
      & \left. +\left | J_{j,k}^{0,1}(\t_1, \t_2) \right|+\left | J_{j,k}^{1,1}(\t_1, \t_2) \right|  \right]
      + \CO(1) n^{2 \max \{\a,\b\}},
\end{align*}
where $J_n^{i,j}$ are defined by
\begin{align*}
  J_{j,k}^{i,j}(\t_1,\t_2) =& k_n^{(\a,\b), i,j}(\cos \t_1, \cos \t_j)  k_n^{(\a,\b),i,j}(\cos \t_2, \cos \t_k) \\
      & +    k_n^{(\a,\b), i,j}(\cos \t_1, \cos \t_k)  k_n^{(\a,\b),i,j}(\cos \t_2, \cos \t_j).
\end{align*}
Hence, as in the proof of Corollary \ref{LebesgueL}, we can reduce the estimate to 
the product of $\Lambda_n^{(i,j)}$, so that the desired result follows from \eqref{Lambda-ij}.
\end{proof}

In the case of $\a = \b = -1/2$, the order of the Lebesgue constant was determined in \cite{BMV}
based on the explicit expression in \cite{X96}. In all other cases, the estimate \eqref{LebesgueCLn} 
is new. One interesting question is if the result can be extended to the case of 
$\max \{\a,\b\} < - \frac12$. We expect that it can be and, furthermore, we believe that the 
order is $\|\CL_n^{\a,\b}\|_\infty =  \CO(1) (\log n)^2$ for $\max \{\a,\b\} < - \frac12$.


\begin{thebibliography}{99}

\bibitem{BP}
      B. Bojanov and G. Petrova,
       On minimal cubature formulae for product weight function,
       \textit{J. Comput. Appl. Math.} \textbf{85} (1997), 113 -121.

\bibitem{BCMV}
      L.~Bos, M.~Caliari, S.~De Marchi and M.~Vianello,
      A numerical study of the Xu polynomial interpolation formula in
      two variables. \textit{Computing} {\bf 76\/} (2005), 311-324.
 
\bibitem{BMV}
       L. Bos, S. De Marchi and M. Vianello,
       On the Lebesgue constant for the Xu interpolation formula.
       \textit{J. Approx. Theory} \textbf{141} (2006), 134-141.

\bibitem{DaiX}
        F. Dai and Y. Xu, 
        Ces\`aro means of orthogonal expansions in several variables,  
        \textit{Constr. Approx.} \textbf{29} (2009), 129--155.

\bibitem{DX}
        C. F. Dunkl and Y. Xu,
         \textit{Orthogonal Polynomials of Several Variables}
         Encyclopedia of Mathematics and its Applications \textbf{81},
         Cambridge University Press, Cambridge, 2001.         

\bibitem{K74}% \index{A}{Koornwinder, T. H.}%
	T. H. Koornwinder,  
	Orthogonal polynomials in two variables  which are eigenfunctions of 
	two algebraically independent partial differential operators, I, II,
        \textit{Proc. Kon. Akad. v. Wet., Amsterdam} \textbf{36} (1974).
        48--66.
 
\bibitem{K75}%\index{A}{Koornwinder, T. H.}%
	T. H. Koornwinder, 
	Two-variable analogues of the classical orthogonal polynomials, in 
	\textit{Theory and applications of special functions}, 435--495, ed. 
        R.~A. Askey, Academic Press, New York, 1975.  

\bibitem{KS}%\index{A}{Sprinkhuizen-Kuyper, I.}%
	T. H. Koornwinder and I. Sprinkhuizen-Kuyper, 
	Generalized power series expansions for a class of orthogonal
 	polynomials in two variables, 
	\textit{SIAM J. Math. Anal.} \textbf{9} (1978), 457--483.

\bibitem{LSX}
        H. Li. J. Sun and Y. Xu, 
        Cubature formula and interpolation on the cubic domain,
        \textit{Numer. Math. Theory Methods Appl.} \textbf{2} (2009), 119--152. 

\bibitem{M}
        H. M\"oller,
        Kubaturformeln mit minimaler Knotenzahl,
	\textit{Numer. Math.} \textbf{ 25} (1976), 185--200.

\bibitem{MP}
         C. R. Morrow and T. N. L. Patterson,
         Construction of algebraic cubature rules using polynomial ideal theory,
         \textit{SIAM J. Numer. Anal.}, \textbf{15} (1978), 953-976.

\bibitem{My}%\index{A}{Mysovskikh, I. P.}%
	I. P. Mysovskikh, 
	\textit{Interpolatory cubature formulas},  Nauka, Moscow, 1981.
%\bibitem{SzV}
%        J. Szabados, and P. V\'ertesi, 
%         \textit{Interpolation of functions}, World Sci. Publ. Co., Teaneck, NJ, 1990.

\bibitem{SX}%\index{Xu, Y.}%
	H. J. Schmid and Y. Xu, 
	On bivariate Gaussian cubature formula,
	\textit{Proc. Amer. Math. Soc.} \textbf{122}  (1994), 833--842. 

\bibitem{S} %\index{A}{Sprinkhuizen-Kuyper, I.}%
	I. Sprinkhuizen-Kuyper,  
	Orthogonal polynomials in two variables. A further analysis of the 
	polynomials orthogonal over a region bounded by two lines and a 
	parabola, 
	\textit{SIAM J. Math. Anal.} \textbf{7}  (1976), 501--518.
	
\bibitem{St}
        A. H. Stroud, 
        {\it Approximate calculation of multiple integrals}, 
        Prentice-Hall, Inc., Englewood Cliffs, N.J., 1971.

\bibitem{Sz}
        G. Szeg\H{o}, 
        {\it Orthogonal polynomials}, 4th ed.
         Amer. Math. Soc. Providence, R.I., 1975.

\bibitem{SV}
        L. Szili and P. V\'ertesi, 
        On multivariate projection operators. 
        \textit{J. Approx. Theory} \textbf{159} (2009), 154 - 164.
   
\bibitem{Xu92}%\index{A}{Xu, Y.}%
	Y. Xu, 
        Gaussian cubature and bivariable polynomial interpolation,
        \textit{Math. Comput.} \textbf{59}  (1992), 547--555.

\bibitem{X94}
        Y. Xu,
        \textit{Common zeros of polynomials in several variables and higher     
        dimensional quadrature},
	Pitman Research Notes in Mathematics Series, Longman, Essex,
        1994.

\bibitem{X96} 
	Y. Xu,
	Lagrange interpolation on Chebyshev points of two variables, 
	\textit{J. Approx. Theory} \textbf{87}  (1996), 220--238.

\bibitem{X97} 
	Y. Xu,
         On orthogonal polynomials in several variables,  {\it Special
         Functions, $q$-series and Related Topics}, The Fields Institute 
         for Research in Mathematical Sciences, Communications Series, 
         Volume 14, 1997, p. 247-270.
         
\bibitem{X10} 
	Y. Xu,
	Orthogonal polynomials and expansions for a family of weight functions in two variables, 
	submitted, arXiv:1012.5268.
	
\end{thebibliography}
\end{document}